\documentclass[a4paper, reqno]{amsart}
\usepackage{a4,wasysym}
\usepackage{amsmath,graphicx,amssymb,color,ulem,bbm}

\newtheorem{theorem}{Theorem}
\newtheorem{corollary}[theorem]{Corollary}

\newtheorem{proposition}[theorem]{Proposition}

\newcommand{\ud}{\mathrm{d}}

\begin{document}

\title[Analysis of a Clonal Evolution Model]{Mathematical Analysis of a Clonal Evolution Model of Tumour Cell Proliferation}
\author[J. Z. Farkas]{J\'{o}zsef Z. Farkas}

\address{Division of Computing Science and Mathematics \\ University of Stirling \\ Stirling, FK9 4LA, United Kingdom }

\email{jozsef.farkas@stir.ac.uk}

\thanks{This work was completed with the support of the Royal Society.}
\author[G. F. Webb]{Glenn F. Webb}

\address{Department of Mathematics \\ Vanderbilt University \\ 1326 Stevenson Center \\
Nashville, TN 37240-0001, USA}
\email{glenn.f.webb@vanderbilt.edu}

\date{\today}

\dedicatory{Dedicated to Professor Jan Pr\"{u}{\ss} on the occasion of his 65th birthday.}

\begin{abstract}
We investigate a partial differential equation model of a cancer cell population, which is structured with respect to age and telomere length of cells. We assume a continuous telomere length structure, which is applicable to the clonal evolution model of cancer cell growth. This model has a non-standard non-local boundary condition. We establish global existence of solutions and study their qualitative behaviour. We study the effect of telomere restoration on cancer cell dynamics. Our results indicate that without telomere restoration, the cell population extinguishes. With telomere restoration, exponential growth occurs in the linear model. We further characterise the specific growth behaviour of the cell population for special cases. We also study the effects of crowding induced mortality on the qualitative behaviour, and  the existence and stability of steady states of a nonlinear model incorporating crowding effect. We present examples and extensive numerical simulations, which illustrate the rich dynamic behaviour of the linear and nonlinear models.
\end{abstract}

\maketitle
\section{Introduction}

Mathematical models of tumour growth provide insight into the dynamic cha\-rac\-teristics of tumour cell populations. Important issues are  cell proliferation, cell heterogeneity, and cell differentiation. These issues are currently examined in two hypothesized models of tumour evolution: the cancer stem cell  (CSC) model and the clonal evolution (CE) model\cite{GCW2009,N1976}. Both models have scientific support, as well as therapeutic implications, and in fact, both models may be involved in the development of a tumour. The essential distinction of the two models is the role of self-renewal in specific cells, and the fraction of the total cell populations that these cells comprise. Here self-renewal means the ability of a cell to inherit a specific function through an unlimited number of successive cell generations.

The CSC model hypothesizes that a very small sub-population of tumour stem cells generate the entire tumour cell population \cite{RMCW2001,Tan}. In mathematical treatments, these stem cells have the ability to self-renew indefinitely, and through sequential mutations generate all the heterogeneous and differentiated cell types comprising the tumour \cite{Alarcon,Ashkenazi,Deasy,Enderling,Ganguly,Gentry,Marciniak,Michor,Molina,Nakata,Piotrowska,Rodriguez,Stiehl,Tello, Vainstein,Wodarz}. This stem cell population lies at the apex of a hierarchical structure of cell types, and tumour evolution is dependent on their unconstrained self-renewing ability \cite{Johnston}.

The CE model  hypothesizes that a tumour population is composed of multiple genetically identical clones, which have the possibility of mutation, selection, and expansion \cite{Busse,RMCW2001,Stiehl2}. In the CE model all undifferentiated cells have a self-renewing capacity for contributing to the tumour evolution. These cells, however, are not at the apex of a hierarchal tree, but rather dispersed widely throughout the tumour cell population as a large fraction of the total tumour cell count.

A central element of cell self-renewal is the Hayflick limit, which constrains differentiated cell lines to a finite number of divisions \cite{H1965}. As differentiated cells divide, telomeres (nucleotide sequences at the ends of chromosomes) shorten until a critical limit is reached, and further divisions are prohibited \cite{LAFGH1992,SB}. The existence of a mechanism which reverses telomere shortening was predicted several decades ago. The CSC model hypothesizes that cancer stem cells circumvent telomere shortening by using  the enzyme telo\-me\-rase to replace their telo\-me\-res, and thus obtain the ability to divide indefinitely \cite{RMCW2001}. Recently, it has been shown that around 90\% of all types of human cancer exhibit a form of telomerase activation [27]. In the CSC model this property resides in an extremely small sub-population, from which all the differentiated tumour cells derive. In contrast, the CE model hypothesizes that a large number of undifferentiated cells, in clonal sub-populations, possess the ability to restore telomeres, and thus sustain the tumour evolution. These two models differ greatly in their fraction of self-renewing cells within the total population. Mathematical models provide a way to compare these self-renewal properties in proliferating cell populations.

In order to model the dynamics of self-renewing cells lineages, which cor\-res\-pond physiologically to chromosomal telomere lengths, is it necessary to track all cells through successive generations. Many mathematical treatments of telomere structure in cell population dynamics have been developed \cite{Arino,Arino2,Bourgeron,DSVBW2007,Kap}. The CSC model has been treated for example in \cite{Kap}, where telomere shortening is investigated in continuum differential equation models. The key focus of these treatments is that mother stem cells produce two daughter cells, one of which is a stem cell, and the other, a differentiated cell, with a limited number of future divisions. These models describe a finite number of sub-populations, each with an idealised precise telomere length. The self-renewal property of a stem cell is captured by one daughter cell remaining in the highest telomere class and the other daughter cell transiting to a lower telomere class. A descent of telomere length shortening continues in each class, with one daughter cell retaining the length of the mother and the other a lower length. This model yields a minute fraction of stem cells at the apex of the total population, consistent with the CSC literature \cite{RMCW2001}. 

The objective of this paper is to develop a mathematical analysis for the alternative CE model and to quantify its dynamic properties.  Our model here incorporates continuum, rather than discrete, telomere lengths in cells. This idealised continuum telomere length is assumed for convenience, to avoid unwieldy compartmentalisation with each telomere length subclass. In our  CE model of telomere shortening there is an unbalanced division of a mother cell to two daughter cells in terms of telomere length. The telomere length of a daughter cell may be less (corresponding to a differentiated cell), or may be equal or greater (corresponding to a self-renewing cell) than the mother cell. In this way the restoration of telomeres and the self-renewal property of cells is distributed through a large proportion of the tumour cell population, consistent with the CE literature\cite{RMCW2001}. The distribution of daughter cell telomere lengths is governed by a ma\-the\-ma\-tically formulated rule that assigns the telomere restoration property to cells which may be viewed as those capable of indefinite divisions in each of the diverse clonal sub-populations.

Our CE model belongs to the class of continuum structured population models, with age and time as dynamic variables, and telomere length as a population structure variable. In the past three decades physiologically structured population models have been increasingly utilised to shed light on some important phenomena of cell populations \cite{Arino,Arino2,Doumic,DSVBW2007,I,Kap,MD}. The power of structuring a population with respect to physiological variables is of great value in understanding the evolution of biological populations. There is  an increasing literature of physiologically structured population models with more than one (physiological) structuring variable. The development of a unified mathematical framework, where structuring variables play substantially different roles, promises to be extremely challenging from the theoretical point of view.

We first consider the following linear model for an age and telomere length structured proliferating cancer cell population.
\begin{align}
\frac{\partial p}{\partial t}(a,l,t)+\frac{\partial p}{\partial a}(a,l,t) & =-(\beta(a,l)+\mu(a,l))p(a,l,t),\quad a\in(0,a_m),\, l\in(0,l_m), \label{lin-eq1} \\
p(0,l,t) & =2\int_0^{l_m}r(l,\hat{l})\int_0^{a_m}\beta(a,\hat{l})p(a,\hat{l},t)\,\ud a\,\ud \hat{l}, \quad l\in (0,l_m), \label{line-eq2} \\
p(a,l,0)& =p_0(a,l),\quad a\in(0,a_m),\, l\in(0,l_m). \label{lin-eq3}
\end{align}
Above $p(a,l,t)$ stands for the density of cells of age $a$, having telomere length $l$ at time $t$. We assume a maximum cell age denoted by $a_m$ and a maximum telomere length denoted by $l_m$. The population count at time $t$ of cells with age between $a_1$ and $a_2$ and telomere length between $l_1$ and $l_2$ is
$$ \int_{a_1}^{a_2} \int_{l_1}^{l_2} p(a,l,t) \, \ud l  \, \ud a,$$
and the total population of all cells at time $t$ is 
$$P(t) =  \int_0^{a_m} \int_0^{l_m} p(a,l,t) \, \ud l  \, \ud a.$$
$\mu(a,l)$ quantifies the natural mortality of cells of age $a$ and telomere length $l$. A mother cell of age $a$ and telomere length $l$ divides into two daughter cells of age $0$ having (possibly) different telomere lengths, at a rate determined by the function $\beta(a,l)$. The function $r(l,\hat{l})$ describes the distribution of daughter telomere lengths $l$ from a mother cell of  telomere length $\hat{l}$. The boundary condition 
(\ref{line-eq2})  accounts for cell division of a mother cell into two daughter cells and requires that 
\begin{equation}\label{scale1eq}
\int_0^{l_m}\int_0^{l_m}r(l,\hat{l})\,\ud \hat{l}\,\ud l=1.
\end{equation}
From a probabilistic interpretation of $r(l,\hat{l})$ it would be natural to normalise the maximal telomere length $l_m$ to $1$. Throughout we retain $l_m$ as a general parameter.
We will impose further regularity assumptions on the model ingredients later on. Note that our model \eqref{lin-eq1}-\eqref{lin-eq3} can be considered as a continuous telomere length structured counterpart of the model recently introduced in \cite{Kap}.

\section{Existence of the governing linear semigroup}

Our starting model \eqref{lin-eq1}-\eqref{lin-eq3} is a linear one, moreover the telomere length $l$ only plays an important role in the somewhat unusual boundary condition \eqref{line-eq2}. Hence to establish the existence of the governing linear semigroup (and therefore the existence of mild solutions of the PDE \eqref{lin-eq1}-\eqref{lin-eq3}) we use a boundary perturbation result due to Greiner, see Theorem 2.3 in \cite{Gre}; see also \cite{Grabosch} and \cite{D-Sch} for similar ge\-ne\-ral results. It is very natural to apply the boundary perturbation result of Greiner, since the unperturbed generator (arising from equation \eqref{lin-eq1} with zero flux boun\-da\-ry condition) is readily shown to generate a translation semigroup. Moreover, it has the added advantage that the spectrally determined growth behaviour of the semigroup follows almost instantaneously, see Proposition 3.1 in \cite{Gre}. For basic definitions and results not introduced in the section we refer the reader to \cite{AGG,NAG}.

To apply the perturbation result of Greiner we set the framework as follows. Assume that $\beta\in C^1_+([0,a_m]\times [0,l_m])$ and $\mu, r\in L^\infty_+((0,a_m)\times (0,l_m))$, and that all of the model parameters are non-negative. In particular, it is natural to assume that $\beta$ does not vanish (in $a$) identically for any $l>0$. We also impose the natural assumption that cells reaching the maximal age do not reproduce anymore, i.e. $\beta(a_m,\cdot)\equiv 0$. This assumption is completely natural from the biological point of view. If cells would still divide upon reaching the maximum age then it would be natural to extend the age-interval beyond $a_m$. In other words, we have set the maximal age $a_m$  such that cells of this age do not divide for any more. On the other hand, we will see later when we introduce nonlinearities in model \eqref{lin-eq1}-\eqref{lin-eq3}, that non-reproducing cells still play a role in the population dynamics, in particular they may affect competition induced mortality. Also note that alternatively, we could have assumed that the mortality is locally integrable with respect to age, but $\int_0^{a_m}\mu(a,\circ)\,\ud a=\infty$ holds, i.e. no individuals survive the maximal age $a_m$. 

For the linear problem \eqref{lin-eq1}-\eqref{lin-eq3}, since we are dealing with density functions, the natural choice of state space is the following Lebesgue space. 
\begin{equation*}
\mathcal{X}=L^1((0,a_m)\times (0,l_m))\cong L^1\left((0,a_m);L^1(0,l_m)\right).
\end{equation*} 
Elements of $\mathcal{X}$ above can be understood as equivalence classes of measurable functions $f(\cdot)(\circ)$ on the square $(0,a_m)\times (0,l_m)$ such that $\int_0^{a_m}\left|f(a)(\circ)\right|\,\ud a\in L^1(0,l_m)$. 
We further set $\mathcal{Y}=L^1(0,l_m)$. We define the operators $\mathcal{A}$ and $\mathcal{B}$ as follows
\begin{align}\label{lin-generator}
\mathcal{A}\, p=-\frac{\partial p}{\partial a},\quad \mathcal{B}\, p=-(\mu+\beta)p,
\end{align}
with   
\begin{equation}\label{op-domain}
D(\mathcal{A})=\left\{p\in W^{1,1}\left((0,a_m);L^1(0,l_m)\right)\right\},\quad D(\mathcal{B})=\mathcal{X}.
\end{equation}
We further introduce the norm 
\begin{equation}
||v||_A=\int_0^{l_m}\int_0^{a_m} |v(a,l)|+|v_a(a,l)|\,\ud a\,\ud l.
\end{equation}
With the $||\cdot ||_A$ norm $D(\mathcal{A})$ is complete, and the maximal operator $\mathcal{A}\,:\,(D(\mathcal{A}),||\cdot||_A)\to\mathcal{X}$ is continuous and linear. Furthermore, we define 
\begin{equation*}
\mathcal{L}\,:\, (D(\mathcal{A}),||\cdot||_A)\to\mathcal{Y}, \quad \mathcal{L}\, v= v(0,\cdot).
\end{equation*} 
Then $\mathcal{L}$ is also continuous and linear, and we have Im$(\mathcal{L})=\mathcal{Y}$. We denote by $\mathcal{A}_0$ the restriction of $\mathcal{A}$ to Ker$(\mathcal{L})$. It is then clear that $\mathcal{A}_0$ generates the strongly continuous and nilpotent shift semigroup $\mathcal{S}$, explicitly given as
\begin{equation}\label{mod-sg}
(\mathcal{S}(t)\,u)(a,l)=\begin{Bmatrix}
u(a-t,l), \quad & a\ge t \\
0, \quad & a<t
\end{Bmatrix}.
\end{equation}
In particular note that for $t>a_m$ we have $(\mathcal{S}(t)\,u_0)\equiv 0$ for any initial condition $u_0\in\mathcal{X}_+$. 
We now define the bounded linear perturbing operator $\Phi\,:\,\mathcal{X}\to\mathcal{Y}$ as follows
\begin{equation}\label{boundary-op}
\Phi(u)=2\int_0^{l_m}r(\cdot,\hat{l})\int_0^{a_m}\beta(a,\hat{l})u(a,\hat{l})\,\ud a\,\ud \hat{l}.
\end{equation}
We also define the corresponding perturbed generator $\mathcal{A}_\Phi$ as 
\begin{equation}
\mathcal{A}_\Phi\, v =\mathcal{A}\, v,\quad D(\mathcal{A}_\Phi)=\{v\in D(\mathcal{A})\,|\, \mathcal{L}\, v=\Phi\, v\}.
\end{equation}
We recall the main result from \cite{Gre} for the reader's convenience.
\begin{theorem}[Greiner]\label{Greiner}
If $\Phi^*(\mathcal{Y}^*)\subseteq D(\mathcal{A}^*_0)$, then $\mathcal{A}_\Phi$ is the generator of a strongly continuous semigroup $\mathcal{T}_\Phi$ on $\mathcal{X}$. 
\end{theorem}
We now apply Theorem \ref{Greiner} to establish the existence of the governing linear semigroup. In our setting we have $\mathcal{X}^*=L^\infty((0,a_m)\times(0,l_m))$, and $\mathcal{Y}^*=L^\infty(0,l_m)$. Furthermore, to compute the adjoint $\Phi^*\,:\,D(\Phi^*)\subset\mathcal{Y}^*\to\mathcal{X}^*$, we note that
\begin{align*}
\langle \Phi\,x,y\rangle= & \,2\int_0^{l_m}y(l)\int_0^{l_m}r(l,\hat{l})\int_0^{a_m}\beta(a,\hat{l})x(a,\hat{l})\,\ud a\,\ud \hat{l}\,\ud l \\
 = & \int_0^{l_m}\int_0^{l_m}\int_0^{a_m}2\,y(l)r(l,\hat{l})\beta(a,\hat{l})x(a,\hat{l})\,\ud a\,\ud l\,\ud \hat{l} \\
 = & \int_0^{l_m}\int_0^{a_m}x(a,\hat{l})\int_0^{l_m}2\,y(l)r(l,\hat{l})\beta(a,\hat{l})\,\ud l\,\ud a\,\ud \hat{l} \\
  & \hspace{72mm} =\langle x,\Phi^*\, y\rangle,
\end{align*}
if we let
\begin{equation}
\Phi^*(y)=2\,\beta(\cdot,\circ)\int_0^{l_m}y(l)r(l,\circ)\,\ud l\in \mathcal{X}^*,\quad D(\Phi^*)=\mathcal{Y}^*.
\end{equation}
Next we note that (see \cite[Sect.III.5]{K}) $g\in D(\mathcal{A}_0^*)$ if there exists an $f\in\mathcal{X}^*$ such that 
\begin{equation}
\langle g,\mathcal{A}_0\, u\rangle=\langle f,u\rangle,\quad \forall\, u\in D(\mathcal{A}_0).
\end{equation}
For any $u\in D(\mathcal{A}_0)$ integration by parts yields
\begin{align}
\int_0^{l_m}\int_0^{a_m} g(a,l)\left(-\frac{\partial u}{\partial a}(a,l)\right)\,\ud a\,\ud l & = \int_0^{l_m}\int_0^{a_m}\frac{\partial g}{\partial a}(a,l)u(a,l)\,\ud a\,\ud l \nonumber  \\
& =\int_0^{l_m}\int_0^{a_m} f(a,l)u(a,l)\,\ud a\,\ud l,
\end{align}
for $g\in \left\{W^{1,\infty}\left((0,a_m);L^1(0,l_m)\right)\,|\,g(a_m,\cdot)\equiv 0\right\}$, and if we let $f=\frac{\partial g}{\partial a}$. Hence it follows from the regularity assumptions on $\beta$ and $r$ that we have 
\begin{equation*}
\Phi^*(\mathcal{Y}^*)\subset \left\{g\in W^{1,\infty}\left((0,a_m);L^1(0,l_m)\right)\,|\, g(a_m,\cdot)\equiv 0\right\}\subseteq D(\mathcal{A}_0^*). 
\end{equation*}
Hence Theorem \ref{Greiner} implies that $\mathcal{A}_\Phi$ generates a strongly continuous semigroup. 

Next we note that for $\lambda\in\rho(\mathcal{A}_0)$, the operator $\mathcal{L}_{|\,Ker(\lambda-\mathcal{A})}$ is a continuous bijection from $\left(ker(\lambda-\mathcal{A}),||\cdot||_A\right)$ onto $\mathcal{Y}$, hence its inverse 
\begin{equation*}
\mathcal{L}_\lambda :=\left(\mathcal{L}_{|\,Ker(\lambda-\mathcal{A})}\right)^{-1}\,: \mathcal{Y}\to\mathcal{X},
\end{equation*} 
is continuous, for $\lambda\in\rho(\mathcal{A}_0)$. In particular, a straightforward calculation shows that 
in our setting we have $\mathcal{L}_\lambda\,:\, y(\circ)\to e^{-\lambda\,\cdot}y(\circ)$, and therefore for $\lambda$ large enough $(\mathcal{I}-\mathcal{L}_\lambda\,\Phi)$ is invertible and positive. Hence by Lemma 1.4 in \cite{Gre} we have that $R(\lambda,\mathcal{A}_\Phi)$ is positive for $\lambda$ large enough. Since  $\mathcal{B}$ is a bounded multiplication operator, we draw the following conclusion.
\begin{corollary}
$\mathcal{A}_\Phi+\mathcal{B}$ generates a positive strongly continuous semigroup of bounded linear operators on $\mathcal{X}$.
\end{corollary}

\section{Asymptotic behaviour}

In this section we study the asymptotic behaviour of solutions of the linear model \eqref{lin-eq1}-\eqref{lin-eq3}. In particular, as we will see, we are going to characterise the spectral bound of the linear semigroup generator $\mathcal{A}_\Phi+\mathcal{B}$ implicitly via the spectral radius of an associated (bounded) integral operator. We will then obtain estimates for the spectral radius of this integral operator.

As we have pointed out earlier one of the advantages of the perturbation theorem of Greiner (Theorem \ref{Greiner}) is that it allows us to establish a desirable regularity property of the semigroup in a straightforward fashion. To this end, for the rest of the paper we further assume that $r$ satisfies the following regularity condition.
\begin{equation}\label{r-reg}
\sup_{l,\hat{l}}\left|r(l+t,\hat{l})-r(l,\hat{l})\right|\le k\,\delta(t),\quad \text{such that}\quad \lim_{t\to 0}\delta (t)=0,\quad k\in\mathbb{R}.
\end{equation}
Note that, for example if $r$ is continuous on the square $[0,l_m]\times [0,l_m]$, then condition \eqref{r-reg} clearly holds.
We recall now from \cite{Gre} the result we are going to apply, for the readers convenience. In particular, Proposition 3.1 in \cite{Gre} reads as follows. 
\begin{proposition}\label{gre-prop}
If $\mathcal{A}_\Phi$ generates a semigroup and $\Phi$ is compact then $\sigma_{ess}(\mathcal{A}_0)=\sigma_{ess}(\mathcal{A}_\Phi)$. In particular, $\rho_+(\mathcal{A}_0)\cap\,\sigma(\mathcal{A}_\Phi)$ contains only poles of finite algebraic multiplicity of the resolvent $R(\lambda,\mathcal{A}_\Phi)$.
\end{proposition}
In the proposition above $\rho_+(\mathcal{A}_0)$ stands for the component of the resolvent set of $\mathcal{A}_0$, which is unbounded to the right. We apply now Proposition \ref{gre-prop} in our setting. In what follows, with a slight abuse of notation, we will denote the semigroup generated by $\mathcal{A}_\Phi+\mathcal{B}$ by $\mathcal{T}_\Phi$.
\begin{proposition}\label{prop}
Assume that \eqref{r-reg} holds. Then the spectrum of the governing li\-ne\-ar semigroup $\mathcal{T}_\Phi$ may contain only elements of the form $e^{t\lambda}$, where $\lambda$ is an eigenvalue of its generator  $\mathcal{A}_\Phi+\mathcal{B}$.
\end{proposition}
\begin{proof}
We note that $\mathcal{B}$ is bounded and $\mathcal{A}_0$ generates a nilpotent semigroup.  Hence, utilising Proposition \ref{gre-prop}, it is only left to show that $\Phi$ is compact. Let $S_{\mathcal{X}}$ denote the unit sphere of $\mathcal{X}$. We have to show that $\Phi(S_{\mathcal{X}})$ is relatively compact in $\mathcal{Y}=L^1(0,l_m)$. Using the Fr\'{e}chet-Kolmogorov criterion of compactness of sets in $L^1$ \cite[Ch.X.1]{Y}, it is enough to show that on the one hand we have
\begin{equation}
\left|\left|\Phi\,S_{\mathcal{X}}\right|\right|_{\mathcal{Y}}=\int_0^{l_m}\left|(\Phi\, x)(l)\right|\,\ud l\le C\,\bar{\beta}\,\bar{r}.
\end{equation}
On the other hand we have 
\begin{align*}
& \int_0^{l_m}\left|(\Phi\,x)(l+t)-(\Phi\, x)(l)\right|\,\ud l \\
 \le & \int_0^{l_m}\int_0^{l_m}2\left|\int_0^{a_m}\beta(a,\hat{l})x(a,\hat{l})\,\ud a\right|\left|r(l+t,\hat{l})-r(l,\hat{l})\right|\,\ud \hat{l}\,\ud l \\
\le & \,2\,l_m\,\bar{\beta}\,k\,\delta(t),
\end{align*}
therefore we have 
\begin{equation*}
\displaystyle\lim_{t\to 0}\int_0^{l_m}\left|(\Phi\,x)(l+t)-(\Phi\, x)(l)\right|\,\ud l=0,
\end{equation*} 
uniformly in $\Phi x$.
\end{proof}
We shall point out that we really needed to utilise Proposition \ref{gre-prop} by Greiner, to obtain that the asymptotic behaviour of the semigroup is determined by the leading eigenvalue of its generator (if it exists). This is due to the distributed structuring with respect to the telomere length of cells, which implies that the governing semigroup is not necessarily eventually differentiable; hence we could not conclude, as for example in \cite{FH} for a classic size-structured models, that the semigroup is eventually compact, and henceforth the spectral mapping theorem holds true. 

As we have seen earlier the semigroup $\mathcal{T}_\Phi$ generated by $\mathcal{A}_\Phi+\mathcal{B}$ is clearly positive, but it may not necessarily be irreducible. To see this, recall from \cite{NAG}, that the semigroup generated by $\mathcal{A}_\Phi+\mathcal{B}$ is irreducible if and only if for every $f$, $0\not\equiv f\in \mathcal{X}_+$, we have that $R(\lambda,\mathcal{A}_\Phi+\mathcal{B})f\gg 0$, i.e. the resolvent is strictly positive, for some $\lambda>s(\mathcal{A}_\Phi+\mathcal{B})$. 

Let $f\in\mathcal{X}_+$, and note that the solution of the resolvent equation \\ $R(\lambda,\mathcal{A}_\Phi+\mathcal{B})f=u$ is
\begin{align}
u(a,\cdot)= & \exp\left\{-\int_0^a\left(\mu(\hat{a},\cdot)+\beta(\hat{a},\cdot)+\lambda\right)\,\ud \hat{a}\right\} \nonumber \\
& \times\left(u(0,\cdot)+\int_0^a\exp\left\{\int_0^s\left(\mu(\hat{a},\cdot)+\beta(\hat{a},\cdot)+\lambda\right)\,\ud \hat{a}\right\}f(s,\cdot)\,\ud s\right). \label{irred1}
\end{align}
Applying the boundary operator $\Phi$ on both sides of equation \eqref{irred1}, it is easily shown that $u(0,\cdot)$ satisfies the inhomogeneous integral equation
\begin{align}
u(0,\cdot)= & 2\int_0^{l_m}u(0,\hat{l})r(\cdot,\hat{l})\int_0^{a_m}\beta(a,\hat{l})\pi(a,\hat{l},\lambda)\,\ud a\,\ud \hat{l} \nonumber \\
& + 2\int_0^{l_m}r(\cdot,\hat{l})\int_0^{a_m}\beta(a,\hat{l})\pi(a,\hat{l},\lambda)\int_0^a\frac{f(s,\hat{l})}{\pi(s,\hat{l},\lambda)}\,\ud s\,\ud a\,\ud \hat{l}. \label{irred2}
\end{align}
Above in \eqref{irred2} we introduced the notation
\begin{equation*}
\pi(a,l,\lambda)=\exp\left\{-\int_0^a\left(\mu(\hat{a},l)+\beta(\hat{a},l)+\lambda\right)\,\ud \hat{a}\right\}.
\end{equation*}
Note that since telomere length is preserved during the lifetime of an individual, it is intuitively clear that the semigroup is irreducible if offspring of all telomere length is produced by some individuals. In other words, if there were individuals of particular telomere lengths who would not produce individuals of any other lengths, then the semigroup would be reducible. We formulate now a rigorous condition for the irreducibility of the semigroup, as this will play an important role later in the qualitative analysis of model \eqref{lin-eq1}-\eqref{lin-eq3}. 
\begin{proposition}
The semigroup $\mathcal{T}_\Phi$ generated by $\mathcal{A}_\Phi+\mathcal{B}$ is irreducible if and only if for any set $I\subset [0,l_m]$ of positive Lebesgue measure, such that its complement set $\bar{I}$ is also a set of positive Lebesgue measure, we have
\begin{equation}
\int_{\bar{I}}\int_I r(l,\hat{l})\,\ud \hat{l}\,\ud l\ne 0.\label{sg-irred}
\end{equation}
\end{proposition}
\begin{proof}
Note that by virtue of the positivity of the semigroup generated by $\mathcal{A}_\Phi+\mathcal{B}$, the resolvent operator $R(\lambda,\mathcal{A}_\Phi+\mathcal{B})$ is positive, for $\lambda$ large enough. Hence the solution $u(0,\cdot)$ of equation \eqref{irred2} is necessarily non-negative almost everywhere. Assume now that \eqref{sg-irred} does not hold, i.e. $\int_{\bar{I}}\int_I r(l,\hat{l})\,\ud \hat{l}\,\ud l=0$ for some sets $I,\,\bar{I}$ of positive measure. Then it is clear that for any $f$ vanishing on $\bar{I}$, equation \eqref{irred2} would admit a solution $u(0,\cdot)$ vanishing on $\bar{I}$, too; and therefore $u$ given by \eqref{irred1} would also vanish on $\bar{I}$. On the other hand, if there was a function $u\not\equiv 0$ vanishing on a set $J\subset [0,l_m]$ of positive measure for almost every $a\in (0,a_m)$, then equation \eqref{irred1} would imply that the solution $u(0,\cdot)$ of \eqref{irred2} would also vanish on $J$, but then this would clearly imply $\int_J\int_{\bar{J}} r(l,\hat{l})\,\ud \hat{l}\,\ud l=0$, a contradiction.
\end{proof}

Proposition \ref{prop} implies that the asymptotic behaviour of solutions of model \eqref{lin-eq1}-\eqref{lin-eq3} is determined by the eigenvalues of the generator (if there are any), hence we study this eigenvalue problem now. In particular, the solution of the eigenvalue problem
\begin{equation}
(\mathcal{A}_\Phi+\mathcal{B})\,\psi=\lambda\,\psi,\quad \psi(0)=\Phi\, \psi, \label{eigv-1}
\end{equation}
is given by
\begin{equation}
\psi(a,l)=\psi(0,l)\exp\left\{-\int_0^a\left(\beta(\hat{a},l)+\mu(\hat{a},l)+\lambda\right)\,\ud \hat{a}\right\}. \label{eigv-2}
\end{equation}
Applying the boundary operator $\Phi$ on both sides of the equation above we have
\begin{equation}
\psi(0,l)=2\int_0^{l_m}r(l,\hat{l})\psi(0,\hat{l})K(\hat{l},\lambda)\,\ud \hat{l}, \label{eigv-3}
\end{equation}
where we defined
\begin{equation}
K(\cdot,\lambda)=\int_0^{a_m}\beta(a,\cdot)\exp\left\{-\int_0^a\left(\beta(\hat{a},\cdot)+\mu(\hat{a},\cdot)+\lambda\right)\,\ud \hat{a}\right\}\,\ud a. \label{eigv-4}
\end{equation}
Hence $\lambda\in\mathbb{C}$ is an eigenvalue of $\mathcal{A}_\Phi+\mathcal{B}$ if and only if, for the given $\lambda$, the integral equation \eqref{eigv-3} has a non-trivial solution $\psi(0,\cdot)$. Then, the eigenvector corresponding to $\lambda$ 
is given by \eqref{eigv-2}. We are in particular interested in the leading eigenvalue (if it exists), which is the spectral bound of the generator, since the semigroup $\mathcal{T}_\Phi$ is positive. This  dominant real eigenvalue, together with the cor\-res\-pon\-ding eigenspace,  determines the asymptotic behaviour of solutions. Also note that equation \eqref{eigv-3} gives naturally rise to define a parametrised family of (positive and bounded) integral operators $\mathcal{O}_\lambda$ as
\begin{equation}
\mathcal{O}_\lambda\, x=2\int_0^{l_m}r(\cdot,\hat{l})K(\hat{l},\lambda)x(\hat{l})\,\ud \hat{l},\quad D(\mathcal{O}_\lambda)=L^1(0,l_m),\quad \lambda\in\mathbb{R}.\label{int-op}
\end{equation}
With this, the characteristic equation \eqref{eigv-3}, which is notably a functional equation, in contrast to a scalar equation in case of a model with age-structure only; can be viewed as an eigenvalue problem for a bounded linear integral operator. 
More precisely, $\lambda$ is an eigenvalue of the generator $\mathcal{A}_\Phi+\mathcal{B}$, if and only the integral operator 
$\mathcal{O}_\lambda$ has eigenvalue $1$. Note that, since $K$ defined in \eqref{eigv-4} is strictly positive, the integral operator $\mathcal{O}_\lambda$ (for every $\lambda\in\mathbb{R}$) is irreducible if and only if condition \eqref{sg-irred} holds \cite[Ch.V]{Sch}. Hence, rightly so, the irreducibility conditions of the semigroup $\mathcal{T}_\Phi$ and the integral operator $\mathcal{O}_\lambda$ coincide. 

Also note that the function $[0,\infty)\ni\lambda\to r(\mathcal{O}_\lambda)$ is continuous and strictly monotone decreasing. These properties can be established by using perturbation results from \cite{AB} and \cite{K}, respectively; see also \cite{CF2012,CP} for similar developments. Also note that if $r$ satisfies condition \eqref{r-reg} then $\mathcal{O}_0$ is shown to be compact exactly in the same way as the operator $\Phi$ was shown to be compact earlier. Hence the spectrum of $\mathcal{O}_0$ may contain only eigenvalues and $0$.

Therefore, in the case when the spectrum of $\mathcal{A}_\Phi+\mathcal{B}$ is not empty, we have the following complete  characterisation of the asymptotic behaviour of solutions of model \eqref{lin-eq1}-\eqref{lin-eq3}.
\begin{enumerate}
\item If $r(\mathcal{O}_0)<1$, then solutions of model \eqref{lin-eq1}-\eqref{lin-eq3} decay exponentially.
\item If $r(\mathcal{O}_0)>1$, then solutions of model \eqref{lin-eq1}-\eqref{lin-eq3} grow exponentially. Moreover, if $r$ satisfies condition \eqref{sg-irred}, then solutions exhibit asynchronous exponential growth.
\item If $r(\mathcal{O}_0)=1$, then for any eigenvector $\psi(0,\cdot)$ corresponding to the spectral radius $1$ of $\mathcal{O}_0$, the function $\psi$ in \eqref{eigv-2} determines a one-parameter family of positive steady states of model \eqref{lin-eq1}-\eqref{lin-eq3}. If $r$ satisfies \eqref{sg-irred}, then there is only one such family of steady states, moreover they are strictly positive, too.
\end{enumerate} 

After the previous general analysis of the asymptotic behaviour of our model next we intend to study the effect of the well-known capacity of telomere restoring of cancer cells on the dynamics. In particular we are going to show, that at least for some general classes of the model ingredients, the telomere length restoring capacity of cancer cells may have a drastic effect on the asymptotic behaviour of solutions already in the linear model \eqref{lin-eq1}-\eqref{lin-eq3}. 
First we note that in the absence of telomere restoring capacity of cells, the function $r$ necessarily vanishes on the half square $\hat{l}\le l$. That is, when a mother cell divides, it only gives birth to daughter cells with shorter telomeres. This is a well-known mechanism observed in healthy cell populations. In particular this telomere shortening of healthy cells results in apoptosis (when reaching the celebrated Hayflick limit) and prevent the possibility of drastic mutations caused by a very large number of iterations of faulty DNA replication. In this case the boundary condition \eqref{line-eq2} can be rewritten as 
\begin{equation}
p(0,l,t)=2\int_l^{l_m}r(l,\hat{l})\int_0^{a_m}\beta(a,\hat{l})p(a,\hat{l})\,\ud a\,\ud \hat{l}, \quad 0<l\le l_m,\,t>0.\label{new-boundary}
\end{equation}
This in particular implies that the eigenvalue problem \eqref{eigv-3} now reads
\begin{equation}
\Psi(l)=2\int_l^{l_m}r(l,\hat{l})\Psi(\hat{l})K(\hat{l},\lambda)\,\ud \hat{l}, \label{new-eigv}
\end{equation}
where we also introduced the notation $\Psi(l):=\psi(0,l)$, for simplicity. Let us assume now that the the function $r$ is separable, i.e. $r(l,\hat{l})=r_1(l)r_2(\hat{l})$ holds for some functions $r_1,r_2$. For example we may assume that $r_1$ is continuously differentiable and $r_2$ is bounded, and that $r_1$ is positive while $r_2$ is non-negative. Then assumption \eqref{r-reg} clearly holds. In this case differentiating equation \eqref{new-eigv} (assuming that an eigenvector $\psi$ with a smooth $\Psi$ component exists) yields the differential equation
\begin{equation}
\Psi'(l)=\Psi(l)\left(\frac{r'_1(l)}{r_1(l)}-2r_1(l)r_2(l)K(l,\lambda)\right),\label{new-eigv2}
\end{equation}
together with the initial condition
\begin{equation}
\Psi(0)=2r_1(0)\int_0^{l_m}r_2(\hat{l})K(\hat{l},\lambda)\Psi(\hat{l})\,\ud \hat{l}.\label{new-eigv3}
\end{equation}
The solution of \eqref{new-eigv2} is
\begin{equation}
\Psi(l)=\Psi(0)\frac{r_1(l)}{r_1(0)}\exp\left\{ -\int_0^l 2r_1(\hat{l})r_2(\hat{l})K(\hat{l},\lambda)\,\ud \hat{l}\right\},\label{new-eigv4}
\end{equation}
which, utilising \eqref{new-eigv3}, leads to the following characteristic equation
\begin{align}
1= & \,2\int_0^{l_m}r_1(l)r_2(l)K(l,\lambda)\exp\left\{-2\int_0^lr_1(\hat{l})r_2(\hat{l})K(\hat{l},\lambda)\,\ud \hat{l}\right\}\,\ud l \nonumber \\
 = & \,1-\exp\left\{-2\int_0^{l_m}r_1(l)r_2(l)K(\hat{l},\lambda)\,\ud \hat{l}\right\}.\label{new-eigv5}
\end{align}
It is clear that equation \eqref{new-eigv5} does not admit any solution $\lambda\in\mathbb{R}$, which, together with the positivity of the semigroup, implies that the spectrum of $\mathcal{A}_\Phi+\mathcal{B}$ does not contain any eigenvalue with a corresponding eigenvector $\psi$ with a continuously differentiable $\psi(0,l)$. On the other hand it is clear from equation \eqref{new-eigv} that any eigenvector is continuous with respect to its second variable. From the biological point of view this phenomenon is associated with the constant loss of telomere length of newborn cells. Indeed, in the absence of telomere restoring capacity we may expect that the cell population accumulates at the minimal length. 
From the mathematical point of view, the non-existence of differentiable eigenvectors is associated with the fact that the governing semigroup cannot shown to be eventually differentiable due to the telomere length structuring.

Next we consider what happens if we account for the telomere restoring capacity of cancer cells. In particular we are going to show that in this case even exponential growth of the cancer cell population is possible. As before, we start with the case of a  separable $r$, i.e. we assume that $r(l,\hat{l})=r_1(l)r_2(\hat{l})$ for some functions $r_1,r_2$. In this case the integral operator $\mathcal{O}_0$ is of rank one, and it is rather straightforward to exactly determine the spectral radius of $\mathcal{O}_0$, which, as we have seen earlier, determines the asymptotic behaviour of the semigroup $\mathcal{T}_\Phi$. In particular, we have
\begin{align}
r(\mathcal{O}_0)= & 2\int_0^{l_m}r_1(l)r_2(l)K(l,0)\,\ud l \nonumber \\
= & 2\int_0^{l_m}r_1(l)r_2(l)\int_0^{a_m}\beta(a,l)\exp\left\{-\int_0^a\left(\beta(\hat{a},l)+\mu(\hat{a},l)\right)\,\ud \hat{a}\right\}\,\ud a\,\ud l.\label{sep-radius}
\end{align}
From \eqref{sep-radius} we can see that depending on the functions $r_1,r_2,\mu$ and $\beta$, the spectral radius $r(\mathcal{O}_0)$ may be greater than $1$. For example in the case of constant functions $r_1,r_2,\mu,\beta$, and setting $l_m=a_m=1$ (note that we can always normalise the maximal age and telomere length), we have 
\begin{equation*}
r(\mathcal{O}_0)=2r_1r_2\frac{1-e^{-(\beta+\mu)}}{\beta+\mu}.
\end{equation*}

To obtain estimates for the spectral radius of $\mathcal{O}_0$ in the more general and difficult non-separable case, note that the Krein-Rutman theorem asserts that if $\mathcal{O}_0$ is compact and positive, and its spectral radius is positive, then it has a positive (not necessarily strictly positive) eigenvector corresponding to its spectral radius. On the other hand de Pagter proved in \cite{Pagter} that if the operator is also irreducible then its spectral radius is strictly positive. As we noted earlier if $r$ satisfies condition \eqref{r-reg} then $\mathcal{O}_0$ is shown to be compact in the same way as the operator $\Phi$.  Assuming now that the spectral radius is positive, let $u\in L^1_+(0,l_m)$ denote the positive eigenvector corresponding to the spectral radius. Then we have
\begin{equation}
2\int_0^{l_m}r(\cdot,\hat{l})K(\hat{l},0)u(\hat{l})\,\ud \hat{l}=r(\mathcal{O}_0)\, u(\cdot),
\end{equation}
which yields
\begin{equation}
2\int_0^{l_m}u(\hat{l})K(\hat{l},0)\int_0^{l_m}r(l,\hat{l})\,\ud l\,\ud \hat{l}=r(\mathcal{O}_0)\int_0^{l_m}u(l)\,\ud l.
\end{equation}
This observation allows us to obtain immediately the following estimates for the spectral radius
\begin{equation}
2\min_{\hat{l}}K(\hat{l},0)\int_0^{l_m} r(l,\hat{l})\,\ud l\le r(\mathcal{O}_0)\le 2\max_{\hat{l}}K(\hat{l},0)\int_0^{l_m}r(l,\hat{l})\,\ud l. \label{radius-est}
\end{equation}
To obtain different estimates, in particular when $\mathcal{O}_0$ may not be irreducible we are going to utilise some minimax principles established in \cite{Marek}. Recall that if $\mathcal{X}$ is a Banach lattice with positive cone $K$, and with dual space $\mathcal{X}^*$ and dual cone $K^*$, respectively; then a set $H'\subseteq K^*$ is called $K$-total if and only if from $\langle x,x'\rangle\ge 0,\,\forall\,x'\in H'$ it follows that $x\in K$. Then for any positive linear endomorphism $\mathcal{O}$ on a Banach lattice $\mathcal{X}$ and for any $x\in K$ one defines
\begin{equation*}
r_x(\mathcal{O})=\sup_{\tau}\left\{\tau\in\mathbb{R}\,|\,(\mathcal{O}\,x-\tau x)\in K\right\}.
\end{equation*}
Recall that by Lemma 3.1 in \cite{Marek} for any $K$-total set $H'\subseteq K^*$ we have
\begin{equation}
r_x(\mathcal{O})=\sup_{\tau}\left\{\tau\in\mathbb{R}\,|\,\langle\mathcal{O}\,x,x'\rangle\ge \tau\langle x,x'\rangle,\,x'\in H'\right\}.
\end{equation}
Moreover, Lemma 3.3 in \cite{Marek} asserts that for any $0\not\equiv x\in K$ we have $r_x(\mathcal{O})\le r(\mathcal{O})$. 
If we let $x\equiv 1$, then we have for any $x'\in K^*$
\begin{align}
 \langle\mathcal{O}_0\,1,x'\rangle & \ge 2\,\min_{l}\int_0^{l_m}r(l,\hat{l})K(\hat{l},0)\,\ud \hat{l}\int_0^{l_m}x'(l)\,\ud l \nonumber \\ 
& =2\,\min_{l}\int_0^{l_m}r(l,\hat{l})K(\hat{l},0)\,\ud \hat{l}\,\langle 1,x'\rangle.\label{r-est1}
\end{align}
Similarly, recall from \cite{Marek} that if we define
\begin{equation}
r^x(\mathcal{O})=\inf_{\tau}\left\{\tau\in\mathbb{R}\,|\,\tau\langle x,x'\rangle\ge \langle\mathcal{O}\,x,x'\rangle,\,x'\in H'\right\},
\end{equation}
then for every $x\in K$ we have $r(\mathcal{O})\le r^x(\mathcal{O})$. Again, choosing $x\equiv 1$, we have for any $x'\in K^*$
\begin{align}
\langle\mathcal{O}_0\,1,x'\rangle & \le 2\,\max_{l}\int_0^{l_m}r(l,\hat{l})K(\hat{l},0)\,\ud \hat{l}\int_0^{l_m}x'(l)\,\ud l \nonumber \\
 & =2\,\max_{l}\int_0^{l_m}r(l,\hat{l})K(\hat{l},0)\,\ud \hat{l}\,\langle 1,x'\rangle.\label{r-est2}
\end{align}
Hence we obtain the following estimates for the spectral radius of $\mathcal{O}_0$
\begin{equation}\label{r-est3}
2\,\min_{l}\int_0^{l_m}r(l,\hat{l})K(\hat{l},0)\,\ud \hat{l}\le r(\mathcal{O}_0)\le 2\,\max_{l}\int_0^{l_m}r(l,\hat{l})K(\hat{l},0)\,\ud \hat{l}.
\end{equation}
Note that the estimates \eqref{radius-est} and \eqref{r-est3} are quite different, in general. 

We next provide hypotheses on $r(l,\hat{l})$ that yield specific growth behavior of the solutions in the presence or absence of highest telomere class renewal. It is known that in the discrete telomere length case, the cell population can have polynomial growth or decay with cells with shortest  telomere length having the highest power growth over time (see e.g. \cite{Arino},\cite{Arino2},\cite{DSVBW2007}). Similar results hold in the continuum telomere length case if we divide the population into telomere length classes. We first consider the case of no self-renewal within any class, that is, all cell divisions result in daughter cells in a shorter telomere length class.

\begin{proposition}\label{propW}
Assume there exists $\delta \in (0,\l_m)$ such that for $\hat{l} \in [0,l_m]$, $r(l,\hat{l}) = 0$ for $0 \leq \hat{l} - \delta \leq l \leq  \hat{l} \leq l_m$. Assume that $\beta_{min} \leq \beta(a,l) \leq \beta_{max}$,  $\mu_{min} \leq \mu(a,l)$,  for all $a \in [0,a_{m}]$, $l \in [0,l_m]$, and $r(l,\hat{l}) \leq r_{max}$, for all $l,\hat{l} \in [0,l_m]$. Let 
$$\sigma = \beta_{min} +\mu_{min}, \quad \omega = 2 \,  \delta \, r_{max} \,  \beta_{max}.$$
Let $p(a,l,t)$ be the solution of (\ref{lin-eq1})-(\ref{lin-eq3}), such that $p_0 \in 
D(\mathcal{A}_\Phi), 
p_0 \geq 0$, and let
$$
P_j(t) = \int^{l_m - j \delta}_{l_m-(j+1) \delta} \,  \bigg( \int^{a_m}_0 \, p(a,l,t) \ud a \bigg) \ud l, \quad j=0,1,\dots,N, \quad 0 <  l_m -  (N+1) \delta.
$$
Then
\begin{equation}\label{W1}
P_0(t) \leq e^{-\sigma t} P_0(0), \quad t \geq 0,
\end{equation}
\begin{equation}\label{W2}
P_1(t) \leq e^{-\sigma t} \bigg(P_1(0) + \omega  \, t \, P_0(0) \bigg), \quad t \geq 0,
\end{equation}
\begin{equation}\label{W3}
P_2(t) \leq e^{-\sigma t} \bigg(P_2(0) + \omega \,  t  \, (P_0(0) + P_1(0)) + \frac{\omega^2 t^2}{2} P_0(0)\bigg), \quad t \geq 0,
\end{equation}
and in general
\begin{equation}\label{W4}
P_j(t) \leq e^{-\sigma t} \bigg(P_j(0) + \sum_{k=1}^j \, \frac{\omega^k t^k}{k!} \, 
\sum_{i=0}^{j-k} 
\binom{j-1-i}{k-1}
P_j(0) \bigg), \quad t \geq 0, \quad j\le N.
\end{equation}
\end{proposition}
\begin{proof}
From (\ref{lin-eq1}) and (\ref{line-eq2}), for $p(a,l,0)  \in D(\mathcal{A}_\Phi)$,
\begin{align}\label{W5}
P_0^{\prime}(t) & = \int^{l_m}_{l_m-\delta}  \int^{a_m}_0 \frac{\partial p}{\partial t} (a,l,t) \ud a  \, \ud l \\ \nonumber
& = \int^{l_m}_{l_m-\delta} \, \int^{a_m}_0  \bigg( - \frac{\partial p}{\partial a} (a,l,t) \, - \, (\beta(a,l) + \mu(a,l))p(a,l,t) \bigg)  \ud a  \, \ud l \\ \nonumber
& = 
\int^{l_m}_{l_m-\delta} \,  \bigg( p(0,l,t) - p(a_m,l,t) - \int^{a_m}_0\, (\beta(a,l) + \mu(a,l))p(a,l,t)  \ud a  \bigg) \, \ud l \\ \nonumber
& \leq \int^{l_m}_{l_m-\delta} \, \bigg(2 \, \int^{l_m}_0 r(l,\hat{l})\int^{a_m}_0  \, \beta(a,\hat{l}) p(a,\hat{l},t) \ud a  \, \ud \hat{l} \bigg) \ud l
- \, \sigma P_0(t) \\ \nonumber 
& = - \, \sigma P_0(t)  \nonumber,
\end{align}
(since $r(l,\hat{l}) \equiv 0$ for $l_m -\delta \leq l \leq l_m$, $0 \leq \hat{l} \leq l_m$ - see Figure \ref{Figure-delta}(a)).
Thus,  $P_0(t) \leq e^{-\sigma t} P_0(0)$.

A similar calculation to (\ref{W5}) yields
\begin{align}\label{W6}
P_1^{\prime}(t) 
& \leq \int^{l_m - \delta}_{l_m - 2 \delta} \, \bigg(2 \, \int^{l_m}_0 r(l,\hat{l})\int^{a_m}_0  \, \beta(a,\hat{l}) p(a,\hat{l},t) \ud a  \, \ud \hat{l} \bigg) \ud l
- \, \sigma P_1(t) \\ \nonumber 
& = \int^{l_m - \delta}_{l_m - 2 \delta} \, \bigg(2 \, \int^{l_m}_{l + \delta} r(l,\hat{l})\int^{a_m}_0  \, \beta(a,\hat{l}) p(a,\hat{l},t) \ud a  \, \ud \hat{l} \bigg) \ud l
- \, \sigma P_1(t) \\ \nonumber 
& = \int^{l_m}_{l_m - \delta} \, \bigg(2 \, \int^{\hat{l} - \delta}_{l_m - 2 \delta } r(l,\hat{l}) \ud l \int^{a_m}_0  \, \beta(a,\hat{l}) \, p(a,\hat{l},t) \ud a  \,  \bigg) \, \ud \hat{l} 
- \, \sigma P_1(t) \\ \nonumber 
& \leq 2 \, \delta \, r_{max} \, \beta_{max} \, \int^{l_m}_{l_m - \delta} \, \bigg( \int^{a_m}_0  \, p(a,\hat{l},t) \ud a  \,  \bigg) \, \ud \hat{l} 
- \, \sigma P_1(t) \\ \nonumber 
& = \omega \, P_0(t) \, -\sigma P_1(t)
\end{align}
(since $r(l,\hat{l}) \equiv 0$ for $l_m - 2 \delta \leq l \leq l_m - \delta$, $0 \leq \hat{l} \leq l + \delta$ - see Figure \ref{Figure-delta}(b)).
Thus,  $P_1(t)$ satisfies (\ref{W2}).

A similar calculation to (\ref{W6}) yields
\begin{align}\label{W7}
P_2^{\prime}(t) 
& \leq \int^{l_m - 2 \delta}_{l_m - 3 \delta} \, \bigg(2 \, \int^{l_m}_0 r(l,\hat{l})\int^{a_m}_0  \, \beta(a,\hat{l}) p(a,\hat{l},t) \ud a  \, \ud \hat{l} \bigg) \ud l
- \, \sigma P_2(t) \\ \nonumber 
& = \int^{l_m - 2 \delta}_{l_m - 3 \delta} \, \bigg(2 \, \int^{l_m}_{l + \delta} r(l,\hat{l})\int^{a_m}_0  \, \beta(a,\hat{l}) p(a,\hat{l},t) \ud a  \, \ud \hat{l} \bigg) \ud l
- \, \sigma P_2(t) \\ \nonumber 
& = \int^{l_m - \delta}_{l_m - 2 \delta} \, \bigg(2 \, \int^{\hat{l} - \delta}_{l_m - 3 \delta } r(l,\hat{l}) \ud l \int^{a_m}_0  \, \beta(a,\hat{l}) \, p(a,\hat{l},t) \ud a  \,  \bigg) \, \ud \hat{l} \\ \nonumber
& \hspace{.5in}  + \, \int^{l_m}_{l_m - \delta} \, \bigg(2 \, \int^{l_m- 2 \delta}_{l_m - 3 \delta } r(l,\hat{l}) \ud l \int^{a_m}_0  \, \beta(a,\hat{l}) \, p(a,\hat{l},t) \ud a  \,  \bigg) \, \ud \hat{l} 
- \, \sigma P_2(t) \\ \nonumber  
& \leq 2 \, \delta \, r_{max} \, \beta_{max} \, \bigg(P_1(t) \, + \,P_0(t) \bigg) - \, \sigma P_2(t) \\ \nonumber 
& \leq \omega \,  e^{-\sigma t} \bigg(P_1(0) + \omega  \, t \, P_0(0) \, + \, P_0(0) \bigg) \, -\sigma P_2(t),
\end{align}
since $r(l,\hat{l}) \equiv 0$ for $l_m - 3 \delta \leq l \leq l_m - 2 \delta$, $0 \leq \hat{l} \leq l + \delta$ - see Figure \ref{Figure-delta}(c).
Thus,  $P_2(t)$ satisfies (\ref{W3}). The general case (\ref{W4}) is proved by induction following similar steps as above.
\end{proof}

\begin{figure}[h!]
\centering
\includegraphics[width=5.1in,height=1.7in]{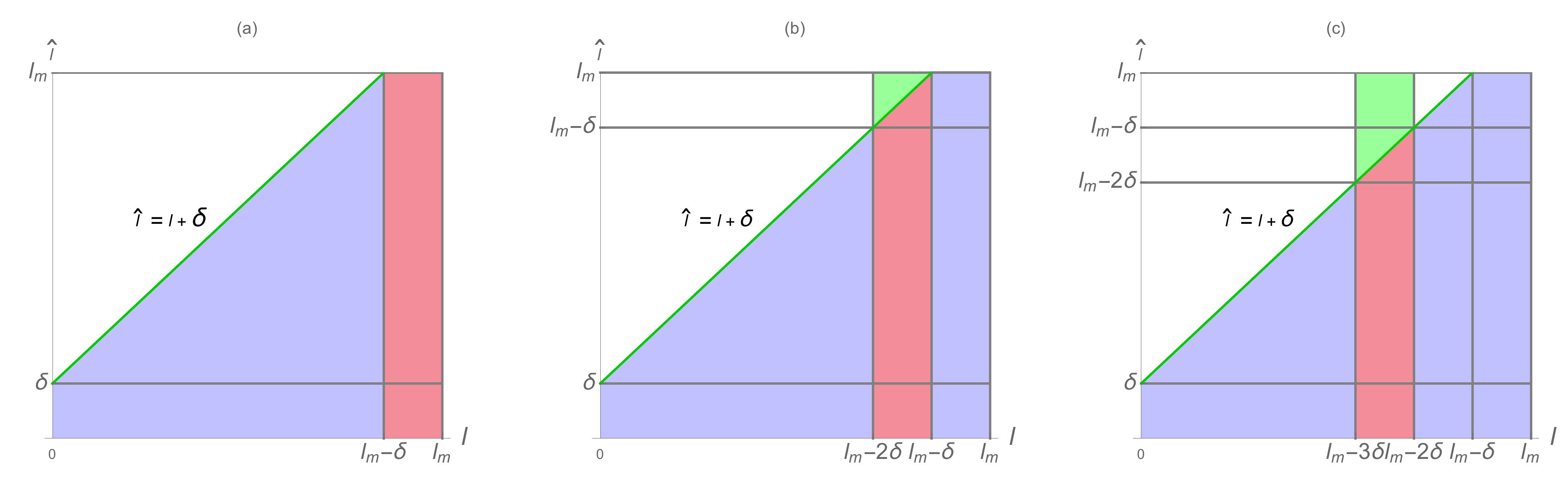}
\caption{The hypotheses on $r(l,\hat{l})$ in Proposition \ref{propW} allow no self-renewal of any telomere length class. $r(l,\hat{l}) \equiv 0$ in the blue regions below the graph $\hat{l} = l + \delta$. $r(l,\hat{l}) \equiv 0$ in the red regions corresponding to classes (a) $P_0(t)$, (b) $P_1(t)$, (c) $P_2(t)$. The green regions correspond to divisions from longer to shorter classes.}.
\label{Figure-delta}
\end{figure}

We next provide sufficient conditions for a class of cells of longest telomeres to have sufficient self-renewal capacity for them to attain proliferative immortality. We assume that the division rate $\beta(a,l)$ and mortality rate $\mu(a,l)$ are constant in this class, and the fraction of dividing cells in this class with self-renewal is sufficiently large to overcome the loss of cells due to division and mortality.
\begin{proposition}\label{propWW}
Let $\delta\in(0,l_m)$, such that $\beta(a,l)\equiv\beta_1>0$ and $\mu(a,l)\equiv\mu_1>0$, for
$l_m - \delta  \leq l \leq l_m$ and $0 \leq a \leq a_m$.  Let $l_m < 2$, and  let $r_1$ be such that
$\int_{l_m - \delta}^{l_m} r(l,\hat{l}) \ud l > r_1$ for $l_m - \delta  \leq \hat{l} \leq l_m$,
and assume that
\begin{equation}\label{r1eq}
(2 r_1 - 1) \beta_1 \geq  \mu_1.
\end{equation}
Let $p(a,l,t)$ be the solution of (\ref{lin-eq1})-(\ref{lin-eq3}), such that $p_0 \in 
D(\mathcal{A}_\Phi)$,  $p_0 \geq 0$. There exists a constant 
$C > 0$ (depending on $p_0$) such that
\begin{equation}\label{W8}
\int_0^{a_m} \int_{l_m - \delta}^{l_m} p(a,l,t) \ud l \, \ud a \geq  C \, e^{(2r_1 \beta_1  - \beta_1 - \mu_1)t}, \quad t \geq 0.
\end{equation}
\end{proposition}
\noindent
(Note that (\ref{scale1eq}) and (\ref{r1eq}) imply that $1/2 < r_1  \leq 1/l_m$, which automatically holds if $l_m$ is normalised to 1. The hypothesis on $r$ means that the fraction of daughter cells with telomere length between $l_m - \delta$ and $l_m$ produced by mother cells in this class is greater than $1/2$.)

\begin{proof}
Let $P_0(a,t) = \int_{l_m - \delta}^{l_m} p(a,l,t) \, \ud l, \, 0 \leq a \leq a_m, \, t \geq 0$.
From (\ref{lin-eq1})-(\ref{lin-eq3})
\begin{equation}\label{W9} 
\frac{\partial P_0}{\partial t}(a,t)+\frac{\partial P_0}{\partial a}(a,t) 
 = - (\beta_1+\mu_1) \, P_0(a,t), 
\end{equation}
\begin{align}
P_0(0,t) & = \int^{l_m}_{l_m-\delta} \, \bigg(2 \, \int^{l_m}_0 r(l,\hat{l})\int^{a_m}_0  \, \beta(a,\hat{l}) p(a,\hat{l},t) \ud a  \, \ud \hat{l} \bigg) \ud l  \nonumber \\
& \geq \int_{l_m - \delta}^{l_m} \, \bigg(2 \, \int^{l_m}_{l_m - \delta} r(l,\hat{l}) \, \ud l \bigg) \, \bigg(\int^{a_m}_0  \, \beta_1 \, p(a,\hat{l},t) \ud a  \bigg) \ud \hat{l}  \nonumber \\
& \geq 2 r_1 \beta_1 \, \int_0^{a_m}  \, P_0(a,t) \, \ud a. \label{W10} 
\end{align}
From the method of characteristics (see e.g. \cite{W})
\begin{equation}\label{P1char}
P_0(a,t) = 
\begin{cases}
P_0(a-t,0) e^{-t \, (\beta_1 + \mu_1)}, \quad a \geq t \\ 
P_0(0,t-a) e^{-a \, (\beta_1 + \mu_1)}, \quad a < t
\end{cases}.
\end{equation}
Let $\hat{P}(a,t)$ satisfy 
\begin{align}\label{W11}
\frac{\partial \hat{P}}{\partial t}(a,t)+\frac{\partial \hat{P}}{\partial a}(a,t) 
& = - (\beta_1+\mu_1) \, \hat{P}(a,t), \\ \nonumber
\hat{P}(0,t) & = 2   r_1 \beta_1 \, \int_0^{a_m} \hat{P}(a,t) \ud a, \\ \nonumber
\hat{P}(a,0) & = P_0(a,0). 
\end{align}
Again from the method of characteristics, $\hat{P}(a,t)$ satisfies
\begin{equation}\label{Phatchar}
\hat{P}(a,t) = 
\begin{cases}
\hat{P}(a-t,0) e^{-t \, (\beta_1 + \mu_1)},  \quad a \geq t \\ 
\hat{P}(0,t-a) e^{-a \, (\beta_1 + \mu_1)}, \quad a < t
\end{cases}.
\end{equation}
From  (\ref{W11}), for $t \geq 0$,
\begin{align}
\hat{P}(0,t) & = 2 \, r_1 \, \beta_1 \, \int_0^{a_m} \hat{P}(a,t) \ud a \nonumber \\
& \, = \, 2 \, r_1 \, \beta_1 \bigg( \int_0^t \, \hat{P}(0,t-a) e^{-a(\beta_1 + \mu_1)} \ud a 
 \nonumber \\
& \hspace{1.0in} + \,  \int_t^{a_m}\, \hat{P}(a-t,0) e^{-t(\beta_1 + \mu_1)} \ud a \bigg)
 \nonumber\\
 & \, \geq \, 2 \, r_1 \, \beta_1 \int_0^t \, \hat{P}(0,t-a) e^{-a(\beta_1 + \mu_1)} \ud a 
 \nonumber \\
& = 2 \, r_1 \, \beta_1  \int_0^t \, \hat{P}(0,b) e^{-(t-b) (\beta_1 + \mu_1)} \ud b, \label{W12}
\end{align}
where the last equality is obtained by a change of the variable of integration.
Let $w(t) = e^{(\beta_1 + \mu_1) t} \, \hat{P}(0,t)$. Then (\ref{W12}) implies
$$
w(t)  \geq 2 r_1  \beta_1 \int_0^t w(a) \ud a,
$$
which implies 
$$ \frac{\ud}{\ud t}\bigg(e^{- 2 r_1 \beta_1 t} \, \int_0^t \, w(a) \ud a \bigg) \geq 0.
$$
Integrating from $a_m$ to $t$ to obtain
$$
e^{- 2 r_1 \beta_1 t} \int_0^t w(a) \ud a \geq e^{- 2 r_1 \beta_1 a_m} \int_0^{ a_m}w(a) \ud a,
$$
which implies
$$
w(t) \geq 2 r_1 \beta_1 e^{2 r_1 \beta_1 (t - a_m)} \int_0 ^{a_m} w(a) \ud a.
$$
Then
\begin{align}\label{W13}
\hat{P}(0,t)  \geq 2 r_1 \beta_1 e^{(2 r_1 \beta_1 - \beta_1 - \mu_1) t} \, e^{- 2 r_1 \beta_1 a_m} \int_0^{a_m} e^{(\beta_1 + \mu_1) a} \hat{P}(0,a) \ud a.
\end{align}
Let $Q(a,t) = P_0(a,t) -  \hat{P}(a,t), \, 0 \leq a \leq a_m, \, t \geq 0$.
Then for $t \geq 0$, from (\ref{W10}) and  (\ref{W12})  
\begin{align}
Q(0,t) & = P_0(0,t)  - \hat{P}(0,t)  \geq  2r_1 \beta_1 \int_0^{a_m} (P_0(a,t) - \hat{P}(a,t) ) \ud a \nonumber \\
& \, = \, 2 \, r_1 \, \beta_1 \bigg( \int_0^t \, (P_0(0,t-a) - \hat{P}(0,t-a)) e^{-a (\beta_1 + \mu_1)} \ud a \nonumber \\
& \hspace{1.0in} + \,  \int_t^{a_m}\, (P_0(a-t,0) - \hat{P}(a-t,0)) e^{- t (\beta_1 + \mu_1)} \ud a \bigg)
\nonumber \\
& = \, 2 \, r_1 \, \beta_1  \int_0^t \, (P_0(0,a) - \hat{P}(0,a)) e^{- (t - a) (\beta_1 + \mu_1)} \ud a
 \nonumber \\
& = 2 \, r_1 \, \beta_1  \int_0^t \, Q(0,a) e^{- (t - a) (\beta_1 + \mu_1)} \ud a. \label{W14}
\end{align}
Then (\ref{W14}) implies
$$
\frac{\ud}{\ud t}\bigg(e^{- 2 r_1 \beta_1 t} \int_0^t Q(0,a) \ud a \bigg) \geq 0,
$$
which integrating from $0$ to $t$ implies $Q(0,t) \geq 0 \, \Longleftrightarrow P_0(0,t) \geq \hat{P}(0,t)$.
Then (\ref{W8}) follows from  (\ref{W13}) and (\ref{P1char}).
\end{proof}

\begin{remark}
We note that a similar result can be established for other classes of cells with self-renewal capability,  of telomere length in a specific $\delta$-interval. Such cell populations can arise from a single mutant cell, which generates more daughter cells with this mutation than competitor cells, and thus expand in a clone within the tumor cell population.
\end{remark}

\section{Incorporating crowding effect}

Next we introduce a nonlinearity in model \eqref{lin-eq1}-\eqref{lin-eq3} by incorporating crow\-ding/competition effects via imposing extra mortality pressure on cells. We follow the same approach as employed previously in \cite{Arino2,DSVBW2007,DVBW2007,Webb1986} for similar cell population mo\-dels. Our equations now read
\begin{align}
\frac{\partial p}{\partial t}(a,l,t)+ & \frac{\partial p}{\partial a}(a,l,t)=-(\beta(a,l)+\mu(a,l))p(a,l,t)-F\left(P(t)\right)p(a,l,t), \label{nonlineq1} \\
p(0,l,t) & =2\int_0^{l_m}r(l,\hat{l})\int_0^{a_m}\beta(a,\hat{l})p(a,\hat{l},t)\,\ud a\,\ud \hat{l},  \label{nonlineq2} \\
p(a,l,0)& =p_0(a,l), \quad P(t)=\int_0^{l_m}\int_0^{a_m}p(a,l,t)\,\ud a\,\ud l. \label{nonlineq3}
\end{align}
In equation \eqref{nonlineq1} $F$ is a non-negative function, and it also satisfies some smoothness assumptions. For example it suffices to assume that $F$ is continuously differentiable. Equation \eqref{nonlineq1} is still semilinear, hence  global existence of solutions can be established at least if $F$ is Lipschitz continuous via integrating along characteristics and using a contraction mapping principle, see for example \cite{CP,W}, where this approach was in fact applied to establish global existence of solutions of a similar model. 
In the simplest case, when $F$ is a linear function, our model \eqref{nonlineq1}-\eqref{nonlineq3} fits into the exact framework studied in \cite{DVBW2007}, if some additional hypotheses are fulfilled. In particular, if we assume that there exists a $\lambda_0\in\mathbb{R}$ and a bounded linear operator $P_0$ (a projection onto the finite-dimensional eigenspace corresponding to the spectral bound of $\mathcal{A}_\Phi+\mathcal{B}$), such that $\displaystyle\lim_{t\to\infty}e^{-\lambda_0\,t}\,\mathcal{T}_\Phi(t)\,x=P_0\,x$ and $F(F_0(P_0(x)))>0$ holds; then the nonlinear semigroup governing \eqref{nonlineq1}-\eqref{nonlineq3} is given explicitly as 
\begin{equation}
\mathcal{S}_\Phi(t)\,x=\frac{\mathcal{T}_\Phi(t)\,x}{1+\int_0^tF(F_0(\mathcal{T}_\Phi(s)\,x))\,\ud s}.
\end{equation}
In the formulas above we introduced the notation $F_0$ for the bounded linear integral operator $F_0 \,  u=\int_0^{l_m}\int_0^{a_m} u(a,l)\,\ud a\,\ud l$, with domain $D(F_0)=\mathcal{X}$. 

Note that some asymptotic results for more general classes of nonlinearities were already obtained in the earlier paper \cite{Webb1986}. In particular for a continuous, non-negative and monotone increasing function $F$ it was proven, under the same assumptions as above, that $\lambda_0<0$ implies that solutions corresponding to initial conditions $x_0\in \mathcal{X}_+\cap D(\mathcal{A}_\Phi+\mathcal{B})$, such that $P_0\, x_0\in\mathcal{X}_+\setminus 0$ holds, tend to $0$. On the other hand, if $\lambda_0\ge 0$, then solutions corresponding to initial conditions as above, tend to $\frac{F^{-1}(\lambda_0)P_0\,x_0}{F_0(P_0\,x_0)}$. 

Here we are mainly interested how the asymptotic behaviour of the nonlinear model changes compared to the linear model \eqref{lin-eq1}-\eqref{lin-eq3}, for a general nonlinear function $F$. In particular, we are interested if the linear model with exponential growth (accounting for the telomere restoring capacity of cancer cells) can be stabilised by adding competition effects, i.e. by incorporating competition induced nonlinearity. Here, under stabilisation, we mean that the addition of the nonlinear mortality term into a linear model whose solutions grow exponentially, leads to a model with a (unique) asymptotically stable positive steady state.

The existence and uniqueness of the positive steady state of the nonlinear model \eqref{nonlineq1}-\eqref{nonlineq3} is established using the techniques we developed in the previous section to study the asymptotic behaviour of the linear model. In particular solving equation \eqref{nonlineq1} for a positive steady state we obtain
\begin{equation*}
p_*(a,l)=p_*(0,l)\exp\left\{-\int_0^a(\beta(r,l)+\mu(r,l))\,\ud r\right\}e^{-aF(P_*)}.
\end{equation*}
Next we substitute this solution into the boundary condition \eqref{nonlineq2} to arrive at an integral equation of the form
\begin{align*}
p_*(0,l)= 2\int_0^{l_m} & r(l,\hat{l})p_*(0,\hat{l}) \\ 
& \times\int_0^{a_m}\beta(a,\hat{l})\exp\left\{-\int_0^a(\beta(r,\hat{l})+\mu(r,\hat{l}))\,\ud r\right\}e^{-aF(P_*)}\,\ud a\,\ud \hat{l}.
\end{align*}
Hence, we define a parametrised family of bounded positive integral operators $\mathcal{Q}_P$ for $P\in [0,\infty)$ as follows
\begin{equation}
\mathcal{Q}_P\,x=2\int_0^{l_m}r(l,\hat{l})x(\hat{l})\int_0^{a_m}\beta(a,\hat{l})\exp\left\{-\int_0^a(\beta(r,\hat{l})+\mu(r,\hat{l}))\,\ud r\right\}e^{-aF(P)}\,\ud a\,\ud \hat{l},
\end{equation}
with domain D$(\mathcal{Q}_P)=L^1(0,l_m)$. Note that for any fixed $P\in [0,\infty)$ we have $\mathcal{Q}_P=\mathcal{O}_{F(P)}$. 
Hence if there exists a $P_*$ such that the integral operator $\mathcal{Q}_{P_*}$ has eigenvalue $1$ with a corresponding normalised positive eigenvector $x$, then 
\begin{equation}
p_*(a,l)=c\,x(l)\exp\left\{-\int_0^a(\beta(r,l)+\mu(r,l))\,\ud r\right\}e^{-aF(P_*)},
\end{equation} 
determines a positive steady state of the nonlinear model \eqref{nonlineq1}-\eqref{nonlineq3}, where the constant $c$ is chosen such that $\int_0^{l_m}\int_0^{a_m}p_*(a,l)\,\ud a\,\ud l=P_*$ holds. Making use of the results we established in the previous section about the spectral radius of the operator $\mathcal{O}$, we summarize our finding in the following proposition.
\begin{proposition}
Assume that $r$ satisfies condition \eqref{sg-irred}. Then, if $F$ is a monotone increasing function, and either of the following conditions holds
\begin{equation}
2\min_{\hat{l}}K(\hat{l},F(0))\int_0^{l_m} r(l,\hat{l})\,\ud l>1,\quad 2\,\min_{l}\int_0^{l_m}r(l,\hat{l})K(\hat{l},F(0))\,\ud \hat{l}>1,\label{mincond}
\end{equation}
the nonlinear model \eqref{nonlineq1}-\eqref{nonlineq3} has a unique strictly positive steady state. 
On the other hand, if $F$ is a monotone decreasing function, then either of the following conditions
\begin{equation}
2\max_{\hat{l}}K(\hat{l},F(0))\int_0^{l_m} r(l,\hat{l})\,\ud l<1,\quad 2\,\max_{l}\int_0^{l_m}r(l,\hat{l})K(\hat{l},F(0))\,\ud \hat{l}<1,\label{maxcond}
\end{equation}
implies that the nonlinear model \eqref{nonlineq1}-\eqref{nonlineq3} has a unique strictly positive steady state. 
\end{proposition}
Note that if $F$ is not monotone, for example if the competition induced mortality exhibits an All\'{e}e-type effect, then we can still establish the existence of the positive steady state by using the estimates \eqref{radius-est}-\eqref{r-est3} for the spectral radius, but we may loose uniqueness in general, and the dynamic behaviour will certainly be more complex.

Next we investigate the stability of the steady states of the nonlinear model \eqref{nonlineq1}-\eqref{nonlineq3}. Note that our model is a semilinear one (moreover the nonlinear operator is differentiable), hence stability can be studied indeed via linearisation, see e.g. \cite{Pruss,W}. To this end note that the linearisation of equation \eqref{nonlineq1} around a steady state $p_*$ reads
\begin{equation}
\frac{\partial u}{\partial t}(a,l,t)+\frac{\partial u}{\partial a}(a,l,t)=-(\beta(a,l)+\mu(a,l)+F(P_*))u(a,l,t)-F'(P_*)p_*(a,l)U(t), \label{linearised}
\end{equation}
where we set $U(t)=\int_0^{l_m}\int_0^{a_m}u(a,l,t)\,\ud a\,\ud l$. The linearised model \eqref{linearised}-\eqref{nonlineq2}-\eqref{nonlineq3} is also governed by a strongly continuous semigroup, since equation \eqref{linearised} is just a bounded perturbation (at least when $F$ is $C^1$) of the linear equation \eqref{lin-eq1}. Moreover, the growth bound of the semigroup is also determined by the spectral bound of its generator. Also note that if $F'(P_*)<0$, for example if $F$ is monotone decreasing, then the semigroup governing the linearised problem \eqref{linearised}-\eqref{nonlineq2}-\eqref{nonlineq3} is positive, too. On the other hand, if $F'(P_*)>0$, then the governing semigroup cannot shown to be positive, and stability might be lost via Hopf-bifurcation. The linearised equation \eqref{linearised} leads the following eigenvalue problem 
\begin{equation}
-u'-(\beta+\mu+F(P_*))u-F'(P_*)\,p_*\overline{U}=\lambda\, u,\quad u(0)=\Phi(u),\label{n-eigv1}
\end{equation}
where we set $\overline{U}=\int_0^{l_m}\int_0^{a_m} u(a,l)\,\ud a\,\ud l$. The solution of the first equation above is
\begin{align}
u(a,l)= & \exp\left\{-\int_0^a\left(\beta(r,l)+\mu(r,l)+F(P_*)+\lambda\right)\,\ud r\right\} \nonumber \\
 & \times \left(u(0,l)-\overline{U}\int_0^a\frac{p_*(r,l)F'(P_*)}{\exp\left\{-\int_0^r\left(\beta(s,l)+\mu(s,l)+F(P_*)+\lambda\right)\,\ud s\right\}}\,\ud r\right).
\end{align}
Imposing the second equation of \eqref{n-eigv1} leads to the following inhomogeneous integral equation
\begin{align}
u(0,l)=  & 2\int_0^{l_m}r(l,\hat{l})u(0,\hat{l})\int_0^{a_m}\beta(a,\hat{l})\Pi(a,\hat{l},\lambda)\,\ud a \,\ud \hat{l} \nonumber \\
 & -2\,\overline{U}\int_0^{l_m}r(l,\hat{l})\int_0^{a_m}\beta(a,\hat{l})\Pi(a,\hat{l},\lambda)\int_0^a\frac{p_*(r,\hat{l})F'(P_*)}{\Pi(r,\hat{l},\lambda)}\,\ud r\,\ud a \,\ud \hat{l}, \label{lin-int-eq}
\end{align}
where we introduced the notation
\begin{equation*}
\Pi(a,l,\lambda)=\exp\left\{-\int_0^a\left(\beta(r,l)+\mu(r,l)+F(P_*)+\lambda\right)\,\ud r\right\}.
\end{equation*}
Hence $\lambda\in\mathbb{C}$ is an eigenvalue of the linearised operator, if and only if the inhomogeneous integral equation \eqref{lin-int-eq} has a non-trivial solution $u(0,\cdot)$. As we can see the information pertaining $\lambda$ is rather implicit. However, in case of the extinction steady state $p_*\equiv 0$, equation \eqref{lin-int-eq} reduces to an integral equation of the form of \eqref{eigv-3}, and therefore the stability of the extinction steady state is established using the techniques we developed in the previous section. Also note that in this case the governing linear semigroup is positive. In particular, we have the following result.
\begin{proposition}
Either of the conditions in \eqref{maxcond} imply that the extinction steady state $p_*\equiv 0$ is asymptotically stable. On the other hand, either of the conditions in \eqref{mincond} imply that the steady state $p_*\equiv 0$ is unstable.
\end{proposition}
\begin{remark}
Note the connection between the existence of a non-trivial steady state and the stability of the trivial one. In particular, for a monotone increasing $F$, either of conditions in \eqref{mincond} implies that a unique strictly positive steady state exists and the trivial one is unstable. On the other hand, if $F$ is monotone decreasing, either of conditions in \eqref{maxcond} implies that a unique strictly positive steady state exists and the trivial steady state is locally asymptotically stable. Moreover, in this case, since the governing linear semigroup is positive, we may anticipate that the unique strictly positive steady state is unstable. In fact we are going to prove this later on.
\end{remark}

Next we study the stability of the positive steady state. First we note that in the special but interesting case, when $F'(P_*)=0$, the eigenvalue problem \eqref{lin-int-eq} (now a homogeneous integral equation) simply reads 
\begin{align}
u(0,l)=  & 2\int_0^{l_m}r(l,\hat{l})u(0,\hat{l})\int_0^{a_m}\beta(a,\hat{l})\Pi(a,\hat{l},\lambda)\,\ud a \,\ud \hat{l}.  \label{new-int-eq}
\end{align}
That is, the eigenvalue problem \eqref{new-int-eq} above can be written as 
\begin{equation*}
\mathcal{O}_{(\lambda+F(P_*))}\,x=1\cdot x,\quad \lambda\in\mathbb{C},\quad x\not\equiv 0,
\end{equation*}
where $\mathcal{O}$ is defined earlier in \eqref{int-op}. Note that, since the semigroup governing the linearised equation is positive, the spectral bound belongs to its spectrum, i.e. it is a dominant real eigenvalue, which determines the asymptotic behaviour.  Also note that, the existence of the positive steady state $p_*$, with total population size $P_*$, requires that $\mathcal{O}_{F(P_*)}\,x=1\cdot x$, with a corresponding positive eigenvector $x$. In the case when $r$ satisfies \eqref{sg-irred}, i.e. the governing linearised semigroup is irreducible, the spectral radius of $\mathcal{O}_{F(P_*)}$ is the only eigenvalue with a corresponding positive (and strictly positive) eigenvector. It is also shown, utilising Proposition A.2 from \cite{AB} that the function $\lambda\to r(\mathcal{O}_{(\lambda+F(P_*))})$ is strictly monotone decreasing, for $\lambda\in [0,\infty)$. Hence we conclude that $r(\mathcal{O}_{F(P_*)})=1$, and therefore $0$ is the dominant eigenvalue of the linearised operator. In this case the governing linear semigroup is strongly stable, but not uniformly exponentially stable, see e.g. \cite[Ch.V]{NAG}. 

Next we consider the general case. We already noted that the information about the spectral values contained in \eqref{lin-int-eq} is rather implicit. Moreover, we note that a biologically relevant and meaningful function $F$ would be of a logistic type, i.e. $F(0)=0$ and $F$ (strictly) monotone increasing. In this case, as we noted before, the governing linear semigroup is not positive, hence we cannot guarantee the existence of a dominant real eigenvalue of the generator, i.e. a dominant real solution $\lambda$ of \eqref{lin-int-eq}. Hence we use a direct approach to establish stability. This approach was employed previously for some simpler structured population models in \cite{FH4,FH8}. The main advantage of this approach is that it does not rely on the positivity of the governing linear semigroup. We now introduce a notation for the generator of the semigroup governing the linearised problem. Let
\begin{align*}
\mathcal{C}_\Phi\, u & =-\frac{\partial u}{\partial a}-(\beta+\mu+F(P_*))u-F'(P_*)\,p_*\,\overline{U}, \\ 
D(\mathcal{C}_\Phi) & =\left\{u\in W^{1,1}\left((0,a_m);L^1(0,l_m)\right)\,|\,u(0,\cdot)=\Phi(u)\right\}=D(\mathcal{A}_\Phi).
\end{align*}
\begin{proposition}\label{stab-prop}
The stationary solution $p_*$ of model \eqref{nonlineq1}-\eqref{nonlineq3} is asymptotically stable if 
\begin{equation}
\mu(\cdot,\circ)+\beta(\cdot,\circ)+F(P_*)> |F'(P_*)|\,P_*+2\beta(\cdot,\circ)\int_0^{l_m}r(\hat{l},\circ)\,\ud \hat{l}\label{stability}
\end{equation}
holds.
\end{proposition}
\begin{proof}
Our goal is to show that there exists a $\kappa>0$ such that the operator $\mathcal{C}_\Phi+\kappa\,\mathcal{I}$ is dissipative ($\mathcal{I}$ stands for the identity). That is, we need to show that there exists a $\kappa>0$, such that we have 
\begin{equation*}
\left|\left|\left(\mathcal{I}-\lambda\,(\mathcal{C}_\Phi+\kappa)\right)w\right|\right|\ge ||w||,\quad \forall\,\lambda>0,\,w\in D(\mathcal{C}_\Phi).
\end{equation*} 
Then, invoking the Lumer-Phillips Theorem \cite[II.3]{NAG}, we obtain that the semigroup generated by $\mathcal{C}_\Phi$ satisfies $||\mathcal{S}_\Phi||\le e^{-\kappa t}$, for $t\ge 0$, i.e. it is exponentially stable. 
To this end, assume that for a given $f\in \mathcal{X}$ and $\kappa\in\mathbb{R}$; $x\in D(\mathcal{C}_\Phi)$ satisfies the equation
\begin{equation}
(\mathcal{I}-\lambda\,(\mathcal{C}_\Phi+\kappa))\, x=f. \label{reseq1}
\end{equation}
Then we are going to show that if condition \eqref{stability} holds, then in fact there exists a $\kappa>0$ small enough, such that $||x||\le ||f||$ holds, for all $\lambda>0$. The main idea, as in \cite{FH4,FH8}, is to divide the interval $(0,a_m)$ into a countable union of subintervals, now for any fixed $l$; on each of which the function $x(\cdot,l)$ is either positive or negative. That is we write 
\begin{equation*}
(0,a_m)=\displaystyle\bigcup_{i} (a_i,b_i)=
\left\{\displaystyle\bigcup_{\hat{i}}(a_{\hat{i}},b_{\hat{i}})\right\}\bigcup\left\{\displaystyle\bigcup_{\bar{i}}(a_{\bar{i}},b_{\bar{i}})\right\},
\end{equation*} 
such that $x(\cdot,l)$ is positive almost everywhere on each subinterval $(a_{\hat{i}},b_{\hat{i}})$, and negative almost everywhere on $(a_{\bar{i}},b_{\bar{i}})$, respectively; and vanishes at each end point except when $a_i=0$ and $b_i=a_m$. 
Over each of the subintervals we multiply equation \eqref{reseq1} by sgn$_a\, x$, we integrate, and we sum up the integrals. Then for any fixed $l$ we obtain the estimate
\begin{align}
\int_0^{a_m}|x(a,l)|\,\ud a \le & \int_0^{a_m}|f(a,l)|\,\ud a+\lambda|x(0,l)|-\lambda F'(P_*)\overline{X}\int_0^{a_m}sgn_a(x)p_*(a,l)\,\ud a \nonumber \\
& +\lambda\int_0^{a_m}(\kappa-(\beta(a,l)+\mu(a,l)+F(P_*))|x(a,l)|\,\ud a.
\end{align}
This, together with  
\begin{align*}
\int_0^{l_m}|x(0,l)|\,\ud l= & \int_0^{l_m}\left|2\int_0^{l_m}r(l,\hat{l})\int_0^{a_m}\beta(a,\hat{l}x(a,\hat{l})\,\ud \,\ud \hat{l}\right|\,\ud l \nonumber \\
\le & \int_0^{l_m}\int_0^{a_m}|x(a,\hat{l})|2\beta(a,\hat{l})\int_0^{l_m}r(l,\hat{l})\,\ud l\,\ud a\,\ud \hat{l},
\end{align*}
yields
\begin{align}
& ||x||\le ||f||+\lambda\int_0^{l_m}\int_0^{a_m}|x(a,l)| \nonumber \\ 
&\times\left(\kappa-(\beta(a,l)+\mu(a,l)+F(P_*))+P_*|F'(P_*)|+2\beta(a,l)\int_0^{l_m}r(\hat{l},l)\,\ud \hat{l}\right)\,\ud a\,\ud l.
\end{align}
Hence if condition \eqref{stability} holds we can indeed choose a $\kappa>0$ such that the solution $x$ of \eqref{reseq1} satisfies $||x||\le ||f||$, for all $\lambda>0$.

To verify that the range condition holds true (see \cite[II.3]{NAG}), note that for any $f\in\mathcal{X}$, the solution of the equation $-\mathcal{A}_\Phi\, u=f-\lambda\,u$ is
\begin{equation}\label{dissip}
u(a,l)=e^{-\lambda a}\left(\Phi(u)+\int_0^a e^{\lambda r}f(r,l)\,\ud r\right),
\end{equation}
where
\begin{equation}\label{dissipative}
\Phi(u)=2\int_0^{l_m}r(l,\hat{l})\Phi(u)\int_0^{a_m}\beta(a,\hat{l})e^{-\lambda a}\,\ud a\,\ud\hat{l}+\Phi\left(\int_0^\cdot e^{-\lambda(\cdot-r)}f(r,\circ)\,\ud r\right).
\end{equation}
Since $\Phi$ is bounded, it follows from the smoothness assumptions we imposed on $\beta$ and $r$ (in particular their boundedness), that for any $f\in\mathcal{X}$ and $\lambda>0$ large enough, the right hand side of \eqref{dissipative} belongs to $L^1(0,l_m)$. 
Therefore, $u$ given by \eqref{dissip} clearly satisfies $u\in D(\mathcal{A}_\Phi)$, and since $\mathcal{C}_\Phi$ is a bounded perturbation of $\mathcal{A}_\Phi$, the range condition holds true, and the proof is completed.
\end{proof}

\begin{remark}
Note that at the extinction steady state the stability condition \eqref{stability} reads
\begin{equation}\label{stability2}
\mu(a,l)+\beta(a,l)+F(0)\ge 2\beta(a,l)\int_0^{l_m}r(\hat{l},l)\,\ud \hat{l}, \quad a \in [0,a_m], \quad l \in [0.l_m],
\end{equation}
This is a biologically relevant and natural condition, as it simply says that if mortality and cell division together is higher  than recruitment of new cells into the population, then the population dies out. Note the connection between (\ref{stability2}) and (\ref{r1eq}), which demonstrates the dichotomy between the stability of the trivial steady state and the exponential growth of a class of cells with longest telomere length.
\end{remark}

Finally, we establish an instability result, for the case when the semigroup governing the linearised equation is positive and irreducible. 
\begin{proposition}\label{instab-prop}
Assume that condition \eqref{sg-irred} holds. Then $F'(P_*)<0$ implies that the positive steady state $p_*$ is unstable.
\end{proposition}
\begin{proof}
We define the operators $\mathcal{C}^0_\Phi$ and $\mathcal{G}$ as follows
\begin{align*}
\mathcal{C}^0_\Phi\, u & =-\frac{\partial u}{\partial a}-(\beta+\mu+F(P_*))u,\quad \mathcal{G}\, u=-F'(P_*)\,p_*\int_0^{l_m}\int_0^{l_m} u(a,l)\,\ud a\,\ud l, \\ 
D(\mathcal{C}^0_\Phi) & =D(\mathcal{C}_\Phi),\quad D(\mathcal{G})=\mathcal{X}.
\end{align*}
Note that if \eqref{sg-irred} holds, then $\mathcal{C}^0_\Phi$ generates a positive irreducible semigroup; moreover, its spectrum is determined by the eigenvalues of its generator $\mathcal{C}^0_\Phi$, which are of finite algebraic multiplicity. Also note that the existence of a (strictly)  positive steady state $p_*$ is characterised by $r(\mathcal{Q}_{P_*})=1$, which is  equivalent to $s(\mathcal{C}^0_\Phi)=0$. 
Since $\mathcal{G}$ is positive and bounded, applying Proposition A.2 from \cite{AB}, we obtain that 
\begin{equation}
0=s(\mathcal{C}^0_\Phi)<s(\mathcal{C}^0_\Phi+\mathcal{G}),
\end{equation}
hence the steady state $p_*$ is unstable.
\end{proof}

\section{Examples and simulation results}

We present three examples to illustrate the asymptotic behaviour of solutions of the CE model. The first example, which is linear, assumes no self-renewal (i.e. telomere restoring capacity) of any cell, and shows an extinction of the cell population. The second example, also linear, allows self-renewal of a large fraction of cells, and shows exponential growth of the cell population. The third example is a nonlinear version of the second example, and shows population growth with stabilization of the total cell count and the age and telomere length structure.
In all of the three examples the age and telomere variables are scaled with $a_m=6$ and $l_m =1$. In all three examples the initial population density is 
\begin{equation*}
p(a,l,0)=1000 \,   l\times\max\{a \, (1-a), \, 0\},
\end{equation*} 
(Figure \ref{Figure1new}). In all three examples the division modulus $\beta \in  C^1(0,a_m)$  is 
\begin{align*}
\beta(a,l) = \begin{Bmatrix}
\max\left\{\beta_0 (a-1) e^{-6(a-1)}, 0\right\} \times \frac{\arctan\left(100 (l-0.5) +\frac{\pi}{2}\right)}{\pi}, \quad
\text{if}  \quad a \geq 1 \\
 0,  \quad \text{if} \quad 0 \leq a<1
\end{Bmatrix},
\end{align*}
where $\beta_0=13$ in Example 1, and $\beta_0=180$ in Examples 2 and 3.
(Figure \ref{Figure1new}). Cells which have telomere length below the critical value $0.5$ have greatly reduced capacity to divide. Note that $\beta(a,l)>0$  for $a>1$, so $\beta(a,l)$ does not vanish in $a$ identically for any $l > 0$. In the  examples the mortality modulus is the constant function  $\mu(a,l) \equiv \mu_0$, where 
$\mu_0 = 0.05$ in Example 1, and $\mu_0=0.3$ in Examples 2 and 3.

\vspace{0.1in}
{\bf Example 1.} No restoration of telomeres occurs in Example 1. The rule governing the telomere length of a daughter cell of length $l$ from a mother cell with length $\hat{l}$  is 
\begin{equation*}
r(l,\hat{l}) = \frac{G(l;\hat{l} - 0.2,0.05)}{0.8},
\end{equation*} 
where $G$ is a Gaussian distribution in $l$ with mean $\hat{l} - 0.2$ and standard deviation $0.05$, and $0.8$ is a normalization factor (Figure \ref{Figure2}). Note that $r(l,\hat{l})$ satisfies \eqref{sg-irred}. The interpretation of this rule is that all daughter cells have telomere length strictly less than their mother cells. The estimates for the spectral radius $r(\mathcal{O}_0)$ in (\ref{radius-est}) and (\ref{r-est3}) are graphed in Figure \ref{Figure3new}. The upper estimates are less than $1$ in both, which means the total population of cells extinguishes. The simulation of the linear model 
\eqref{lin-eq1}-\eqref{lin-eq3} for Example 1 is given in Figure \ref{Figure4new} and Figure \ref{Figure5new}. 

\begin{figure}[h!]
\centering
\includegraphics[width=5in,height=3in]{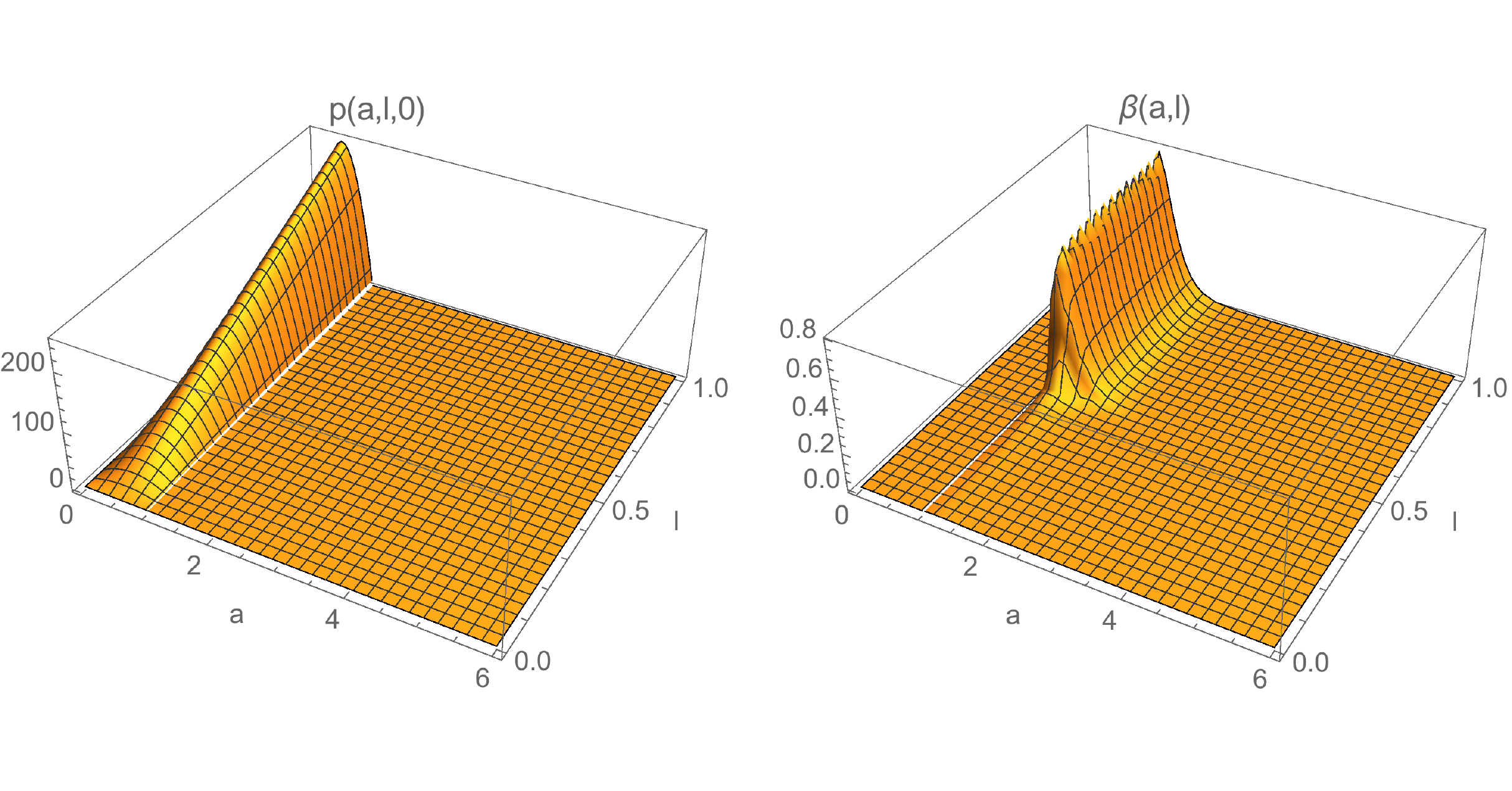}
\caption{In the left panel the initial age and telomere length distribution $p(a,l,0)$ of the cell population in all three examples is plotted. In the right panel the age and telomere length dependent division modulus $\beta(a,l)$ for Example 1 is plotted. No cell divides with $a \leq 1$. Cells with $l<0.5$ have greatly reduced capacity to divide.}
\label{Figure1new}
\end{figure}

\begin{figure}[h!]
\centering
\includegraphics[width=5.3in,height=3.1in]{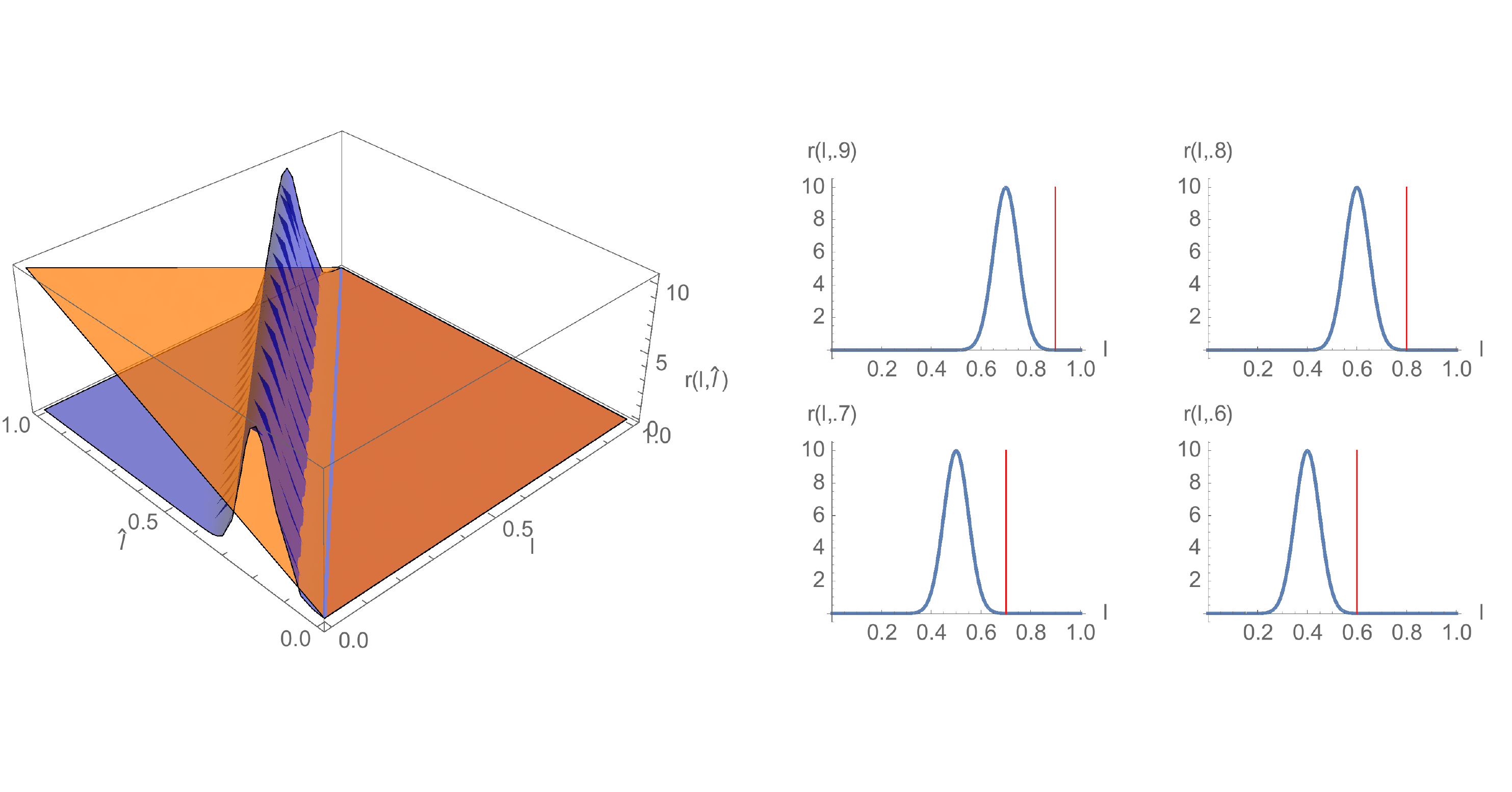}
\caption{In the left panel the blue surface is the graph of $r(l,\hat{l})$ in Example 1. The orange surface is  the graph of $10   \max\{\hat{l} - l,0\}$. In the right panel slices of the graph of $r(l,\hat{l})$ at the values $\hat{l} = 0.9, 0.8, 0.7$ and $0.6$ are plotted. The telomere lengths of daughter cells are all less than the mother cell.}
\label{Figure2}
\end{figure}

\begin{figure}[h!]
\centering
\includegraphics[width=5.3in]{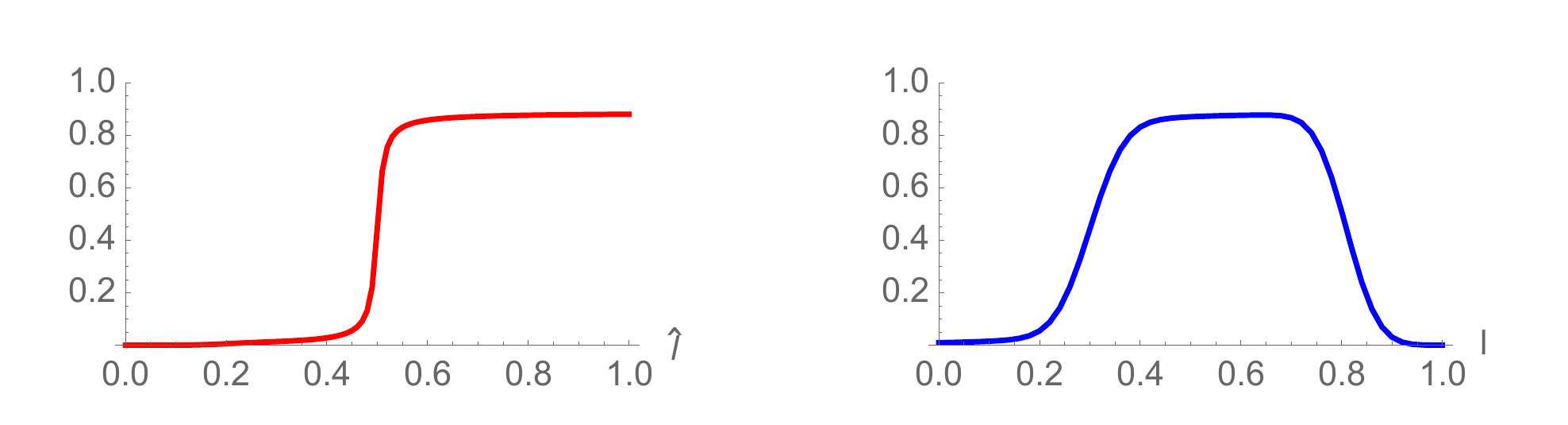}
\caption{The graphs of $2 \, K(\hat{l},0)\int_0^{l_m} r(l,\hat{l}) dl$ (red) and 
$2 \,\int_0^{l_m}r(l,\hat{l})K(\hat{l},0) d\hat{l}$ (blue) for Example 1. Since the maximum of each is less than $1$, (\ref{radius-est}) and (\ref{r-est3}) imply that the spectral radius $r(\mathcal{O}_0)<1$. Thus, the total population of cells converges to $0$.}
\label{Figure3new}
\end{figure}

\begin{figure}[h!]
\centering
\includegraphics[width=5in,height=3in]{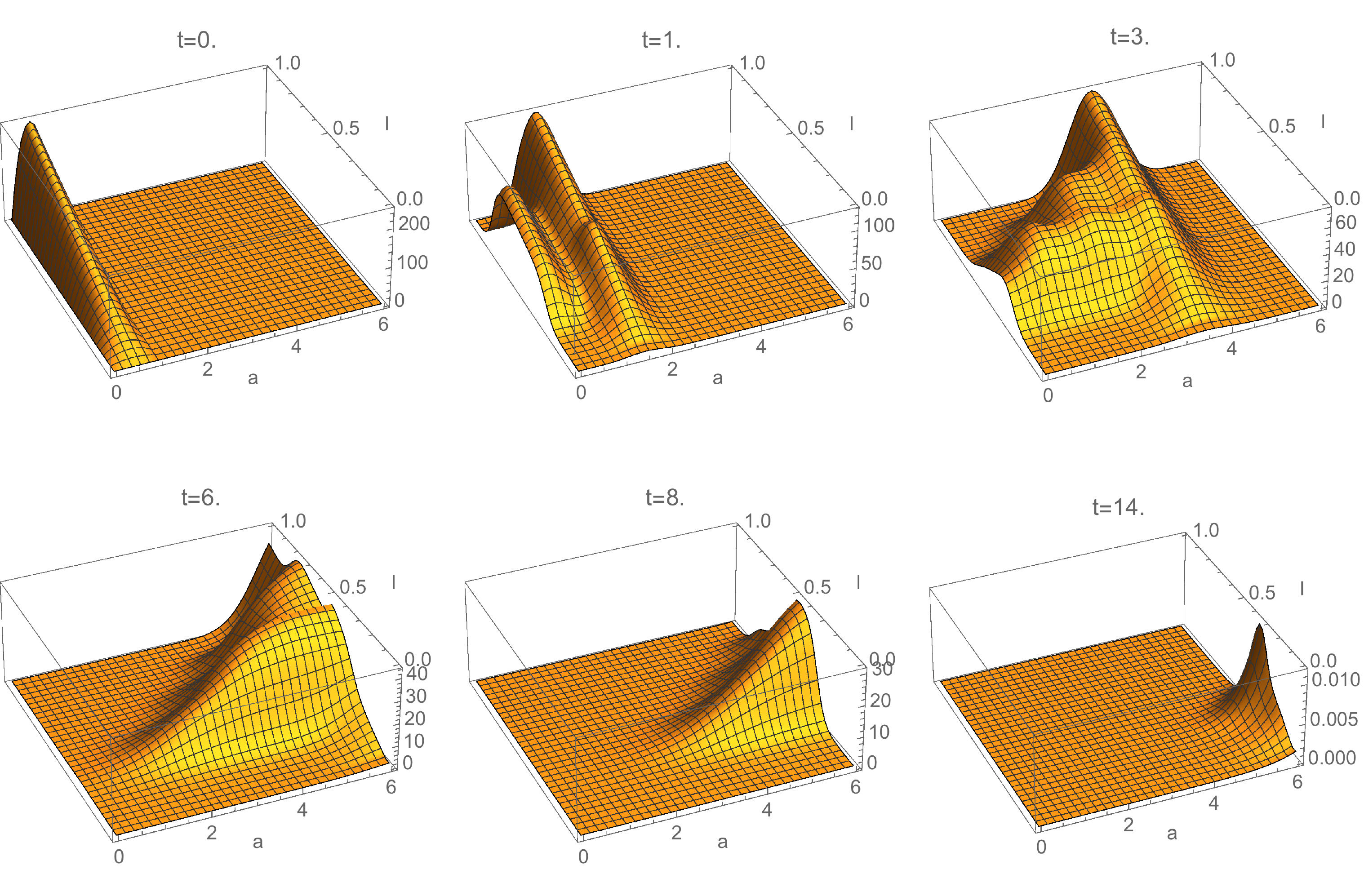}
\caption{The cell population densities $p(a,l,t)$ for Example 1 for time values $t = 0, 1, 3, 6, 8$, and $14$ are plotted.}
\label{Figure4new}
\end{figure}

\begin{figure}[h!]
\centering
\includegraphics[width=11cm]{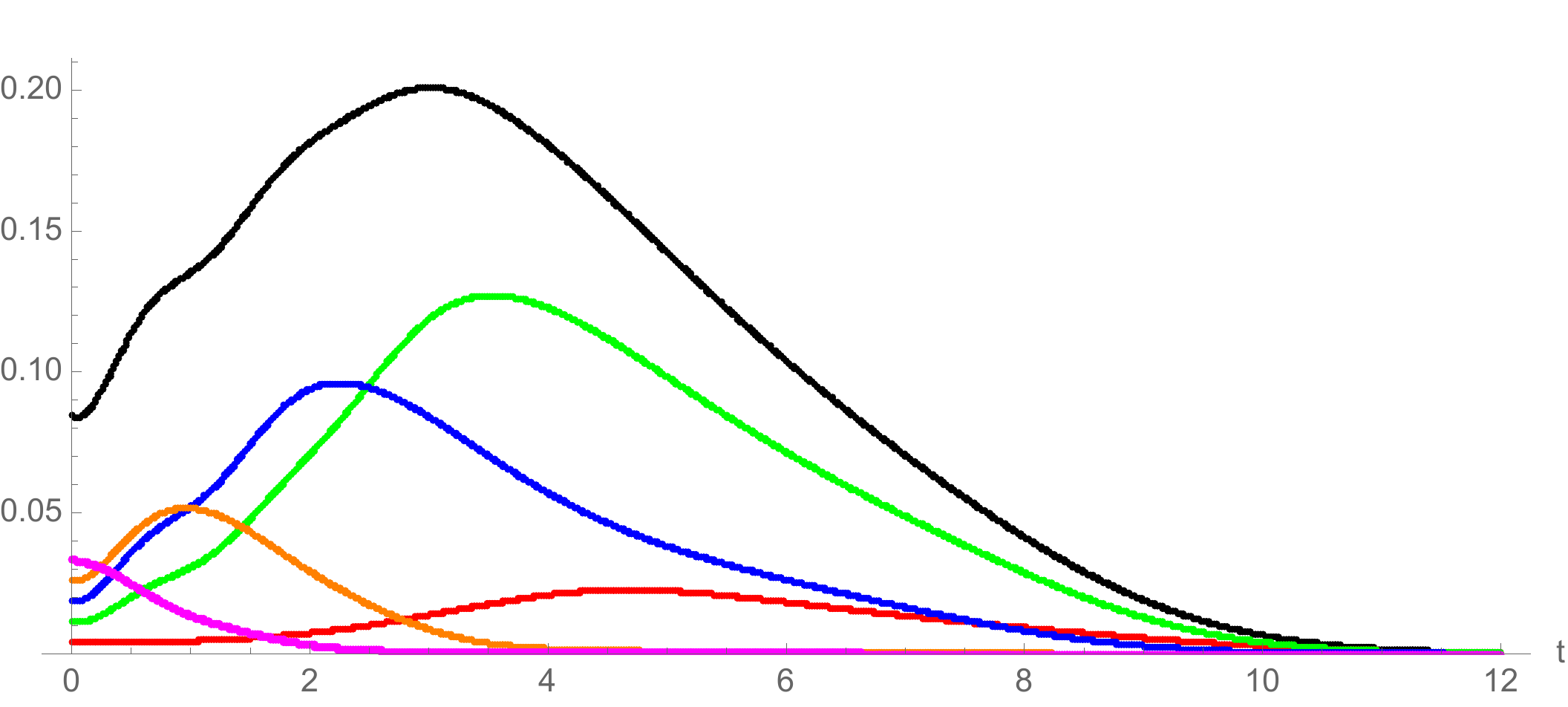}  
\caption{The time plot shows for Example 1 the subpopulations of cells as follows: black - total population;   magenta - telomere lengths between 0.8 and 1.0;  orange - between 0.6 and 0.8; blue - between 0.4 and 0.6; green - between 0.2 and 0.4; red - between 0.0 and 0.2;
all converging to $0$ as time advances.}
\label{Figure5new}
\end{figure}

{\bf Example 2.} In this example  restoration of telomeres occurs in cells with larger telomere lengths. The rule governing the telomere length of a daughter cell of length $l$ from a mother cell with length $\hat{l}$  is 
\begin{equation*}
r(l,\hat{l}) = \frac{G(l;m(\hat{l}),0.05)}{0.5},
\end{equation*} 
where $G$ is a Gaussian distribution in $l$ with mean $m(\hat{l}) = 1+2(\hat{l}-0.9)$ and standard deviation $0.05$; and $0.5$ is a normalization factor (Figure \ref{Figure5} and Figure \ref{Figure6}). Note that $r(l,\hat{l})$ satisfies \eqref{sg-irred}. The interpretation of this rule is that some daughter cells have telomere length equal or greater than  their mother cells, when the mother cells have longer lengths. The simulation of this telomere restoration rule for the linear model
\eqref{lin-eq1}-\eqref{lin-eq3} in Example 2 is given in Figure \ref{Figure8new}. The total population $P_2(t)$ of cells stabilizes in the age and telomere variables, but the total population size grows exponentially (Figure \ref{Figure9new}). A large fraction of cells have longer telomere lengths as the age-telomere length distribution stabilizes.

\vspace{0.4cm}
{\bf Example 3.} Example 3 is the nonlinear version \eqref{nonlineq1}-\eqref{nonlineq3} of Example 2, with the same parameters.  Additionally, the crowding term $F$ in Example 3 is defined as $$F(P)=\gamma P,$$ 
with $\gamma = 0.00001$. The population stabilizes both in structuring variables $a$ and $l$ (see Figure \ref{Figure10new}), as well as in time (Figure \ref{Figure11new}). The self-renewal properties of the longest telomere length cells in Example 3, combined with the nonlinear crowding effect, result in convergence to equilibrium .  As in Example 2, a large fraction of total  cells have longer telomere lengths at the stable steady state.
\begin{figure}[h!]
\centering
\includegraphics[width=4.1in]{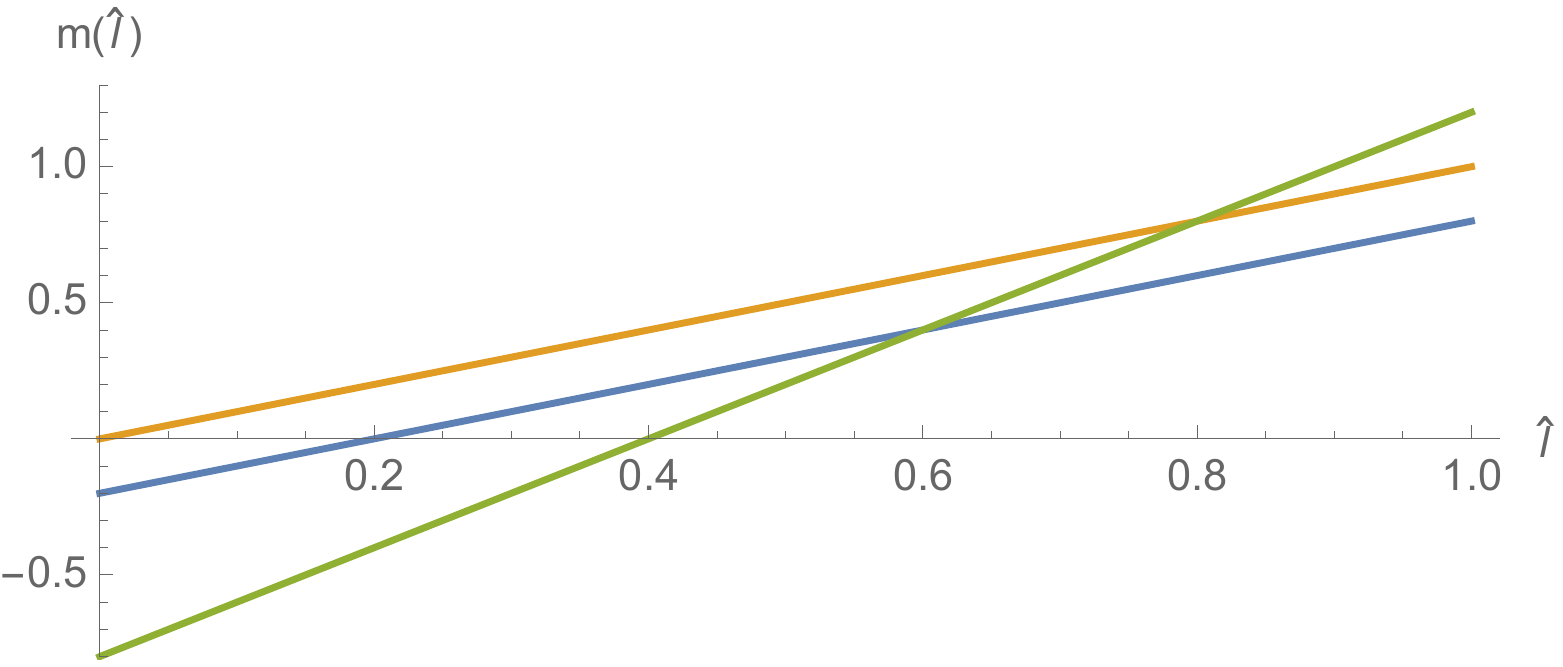}  
\caption{The graphs for the means of the Gaussian distributions in the distribution of telomere rules $r(l,\hat{l})$ are plotted. In Example 1  $m(\hat{l}) = \hat{l} - 0.2$ (blue). In Examples 2 and 3 $m(\hat{l})=1+2(\hat{l}-0.9)$ (green). The orange line is $m(\hat{l}) = \hat{l}$.}
\label{Figure5}
\end{figure}
\begin{figure}[h!]
\includegraphics[width=5in,height=3in]{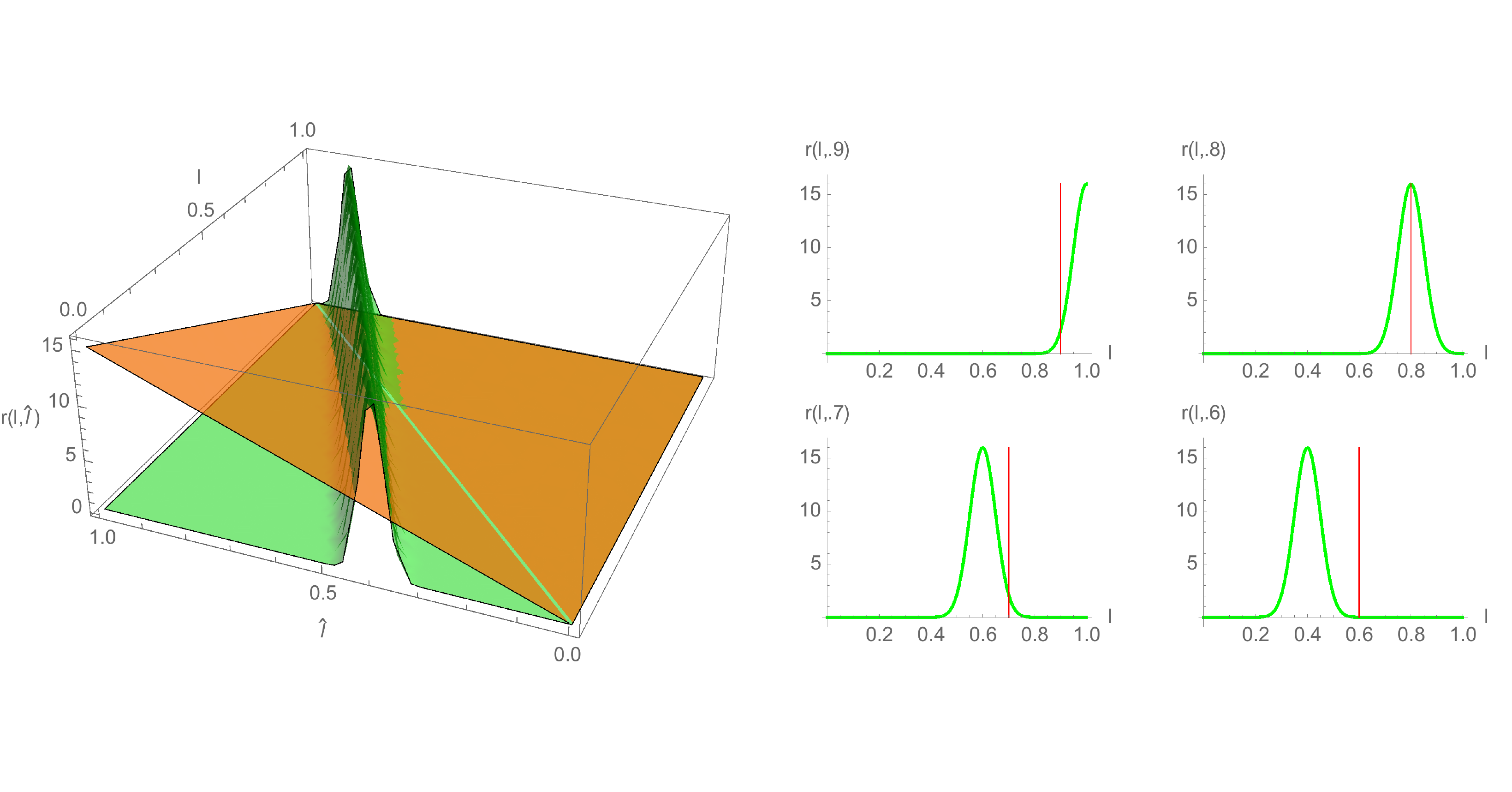}  
\caption{On the left panel the graph of $r(l,\hat{l})$ in Example 2 is plotted in green. The orange surface is  the graph of $10 \max\{\hat{l} - l,0\}$. On the right panel the slices of the graph of $r(l,\hat{l})$ for Example 2 at the values $\hat{l} = 0.9, 0.8, 0.7, 0.6$ are plotted. The telomere lengths of daughter cells from longer length mother cells may be greater than the mother cells. The telomere lengths of daughter cell from shorter length mother cells are all shorter than the mother cells telomere lengths.}
\label{Figure6}
\end{figure}

\begin{figure}[h!]
\includegraphics[width=4.8in,height=2.9in]{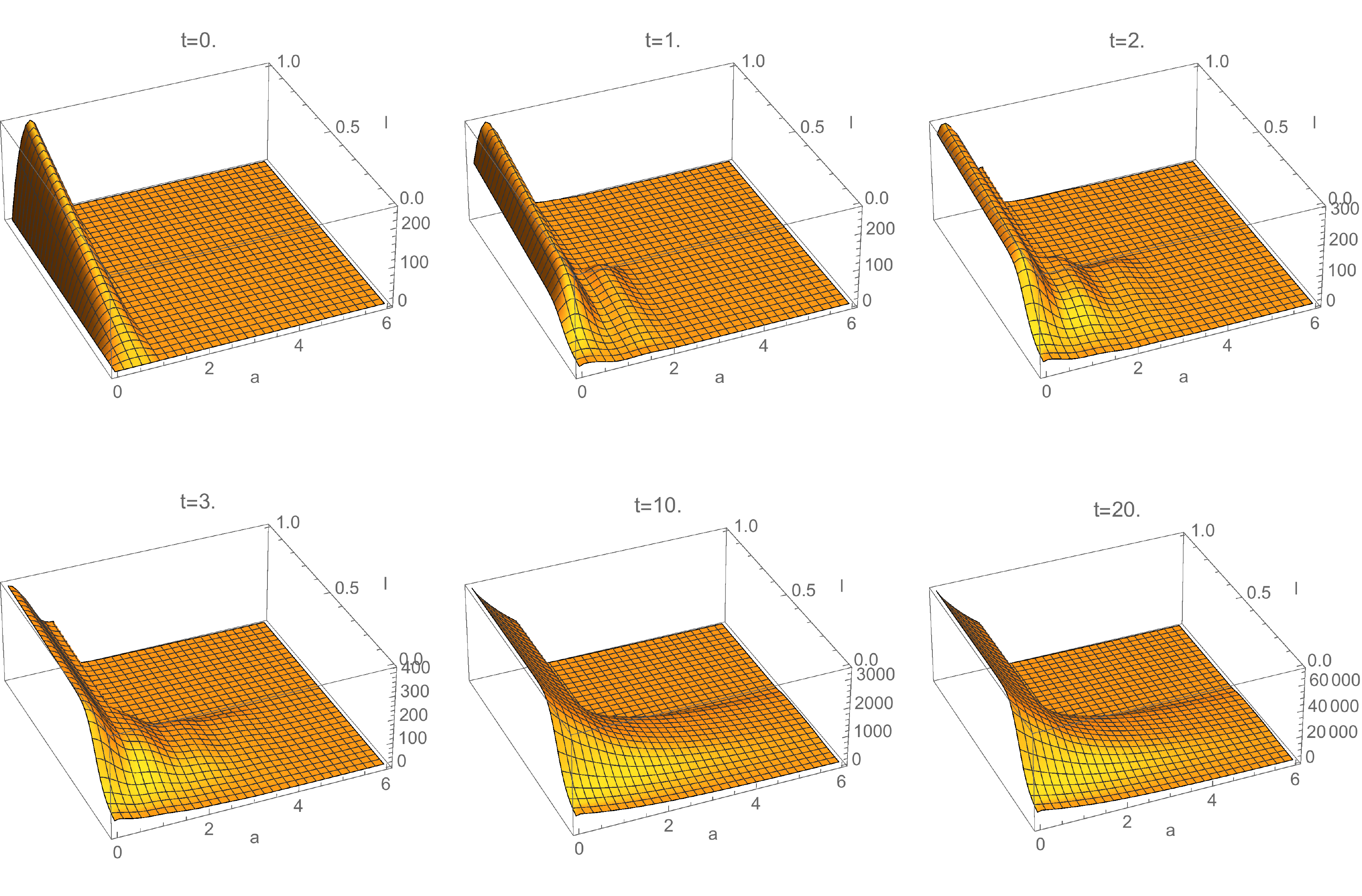}
\caption{The cell population densities $p(a,l,t)$ for Example 2 for the values $t = 0, 0.2, 0.5, 1, 1.5, 2, 10, 20$ are plotted. The population stabilizes with respect to age and telomere length even as  the total population size grows exponentially.}
\label{Figure8new}
\end{figure}

\begin{figure}[h!]
\centering
\includegraphics[width=11cm, height=4.8cm]{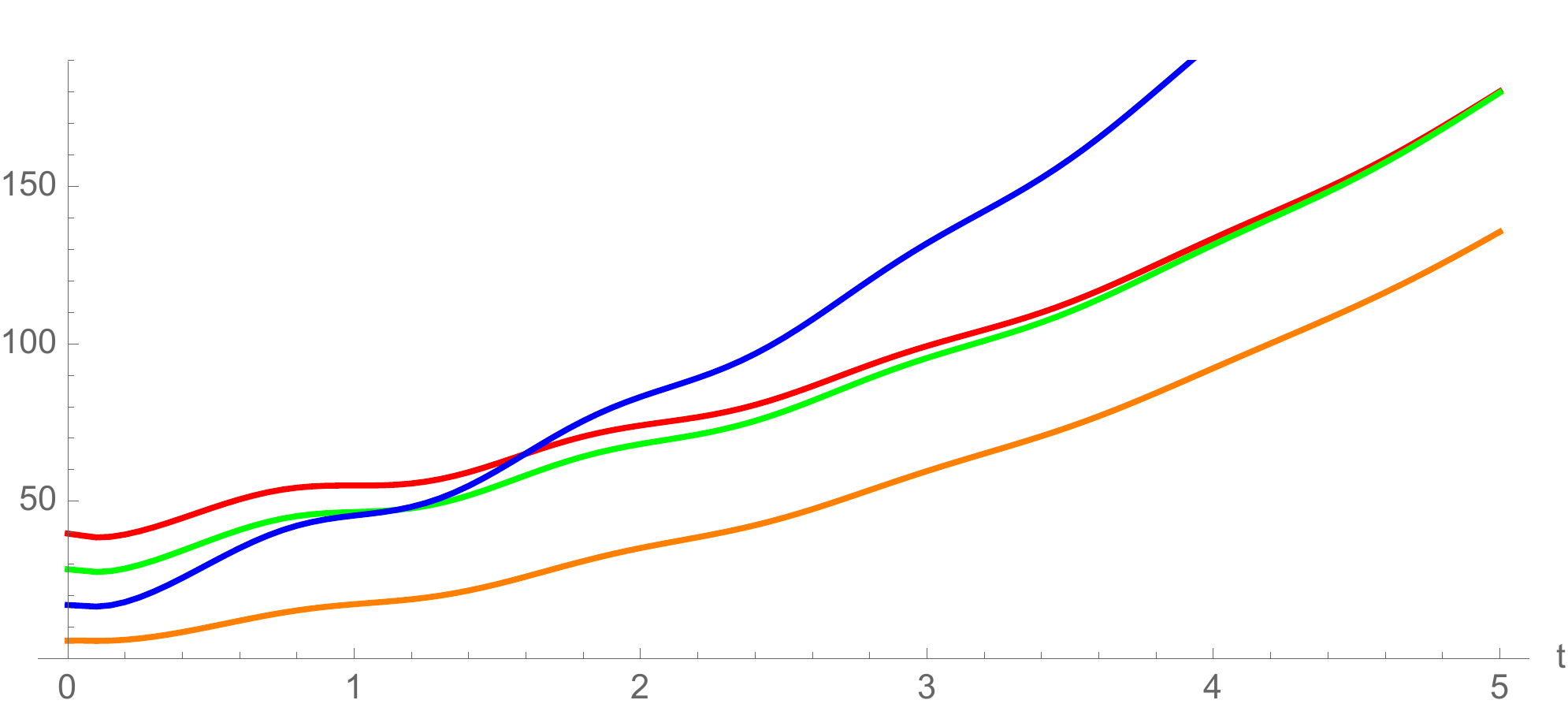}  
\caption{The time plot shows for Example 2 the subpopulations of cells as follows: red - telomere lengths between 0.75 and 1.0;  green - between 0.5 and 0.75; blue - between 0.25 and 0.5; orange - between 0.0 and 0.25;
all growing exponentially as time advances. The total population of cells  (not plotted here) is the sum of these four subpopulations.}
\label{Figure9new}
\end{figure}

\begin{figure}[h!]
\centering
\includegraphics[width=4.8in,height=2.9in]{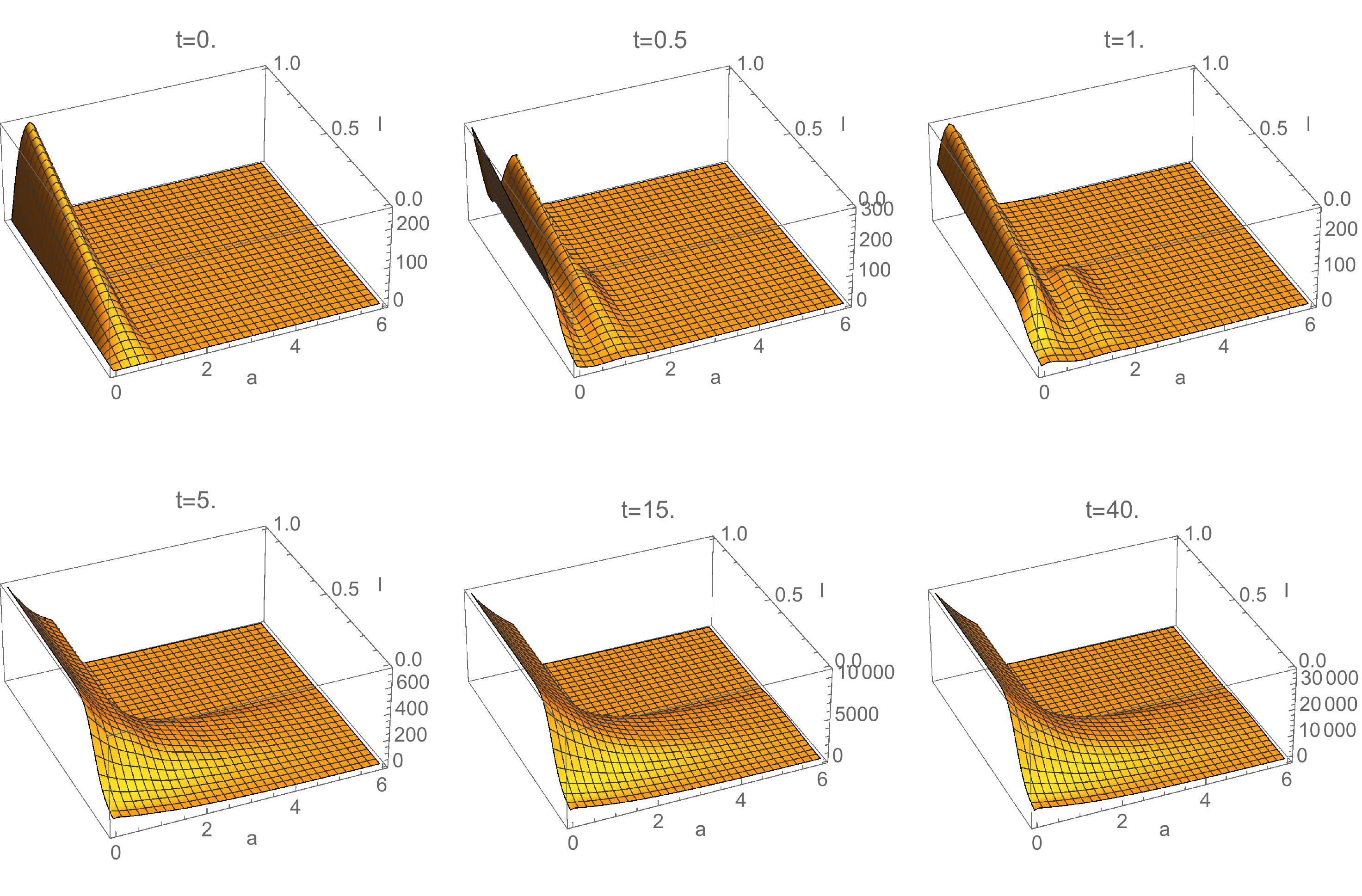}
\caption{The cell population densities $p(a,l,t)$ for Example 3 for the time values $t = 0, 0.5, 1,  5, 40, 50$ are plotted. The population stabilizes with respect to age and telomere length as the total population size stabilizes.}
\label{Figure10new}
\end{figure}

\begin{figure}[h!]
\centering
\includegraphics[width=11cm, height=4.8cm]{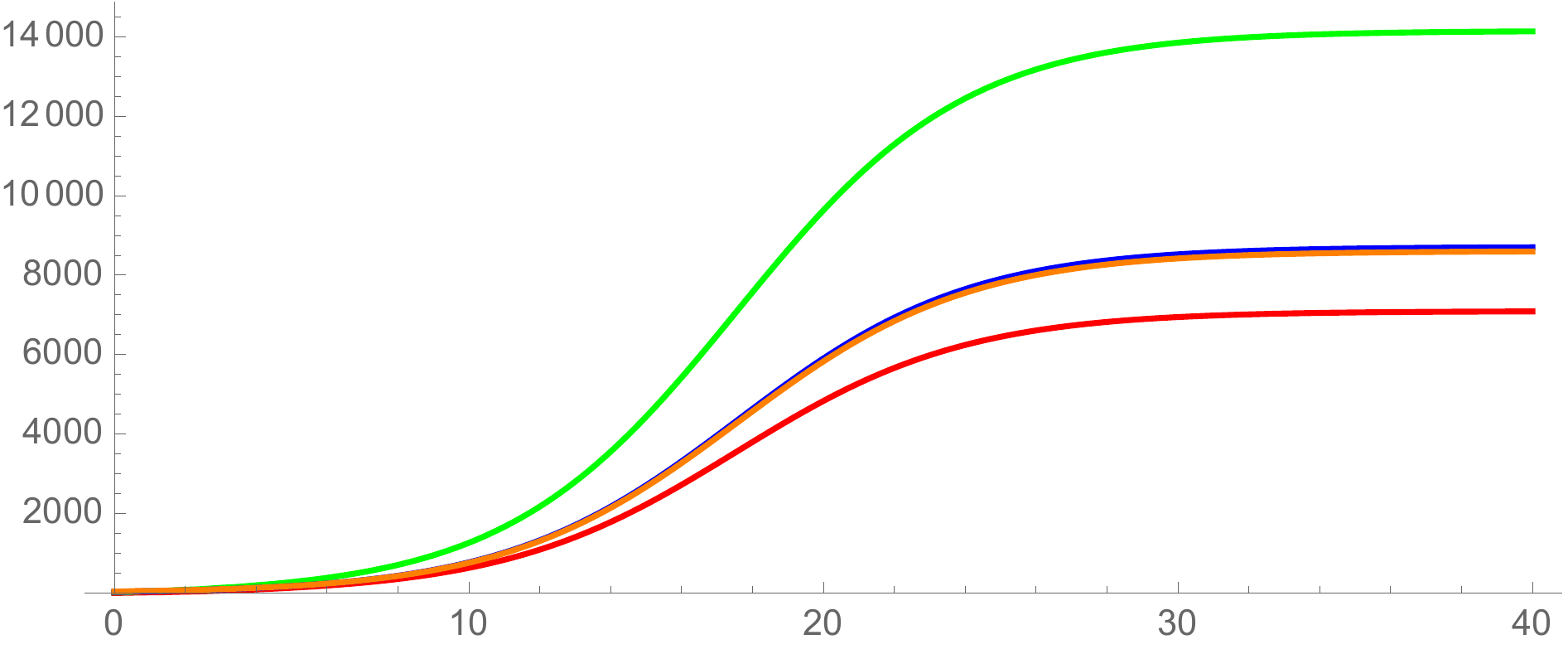}  
\caption{The time plot shows for Example 3 the subpopulations of cells as follows: orange - telomere lengths between 0.75 and 1.0;  blue - between 0.5 and 0.75; green - between 0.25 and 0.5; red - between 0.0 and 0.25; all converging to a steady state value. The total population of cells  (not plotted here) is the sum of these four subpopulations.}
\label{Figure11new}
\end{figure}

\clearpage

\section{Discussion}
In this work we have developed a mathematical formulation of cancer cell self-renewal for the clonal evolution model of tumour growth based on telomere restoration. The model allows for a continuum of telomere lengths, and thus contrasts to mathematical treatments of the cancer stem cell model, which incorporate many discrete telomere length classes \cite{Kap}. 
The cancer stem cell model formulates a hierarchal array of length classes with self-renewing cells in one longest telomere class. The clonal evolution model allows multiple classes of telomere length cells to have self-renewal capacity, corresponding to clonal structuring.

Our model is thus more tractable for analysis and simulations, which we have provided here. In particular, in this work we focused on the effect of telomere restoring capacity of cancer cells. In particular we showed that the asymptotic behaviour of the linear model is determined by the spectral radius of an integral operator. We then obtained estimates for the spectral radius of this integral operator. In Section 4 we extended our model by incorporating a competition induced nonlinearity in the mortality of cells. We treated the existence and stability of steady states of the the nonlinear model by using some well-known results from the theory of positive operators. Finally, in Section 5, we presented a number of examples and the results of numerical simulations, both for the linear and nonlinear model. The simulations highlight the dynamic behaviour of the model and underpin the analytical results obtained. 

Naturally, many issues remain in the mathematical investigation of the clonal evolution model of tumour growth. Important questions which can be addressed in the framework of a mathematical model include the following.  
\begin{itemize}
\item How do sequential mutations enter into the model formulation? 
\item How can quiescent cells be incorporated into the model? 
\item How can spatial heterogeneity be formulated in the equations? 
\item How can more complex nonlinearities be incorporated?
\item How can the model be implemented with actual experimental data? 
\end{itemize}
These issues and the development of a  general mathematical framework for analysing physiologically structured models with additional distributed structuring variables remain important subjects forf further research.

\subsection*{Acknowledgment}
We thank the Royal Society for financial support.


\end{document}